\documentclass{amsart}

\usepackage{amsfonts}
\usepackage{amsmath}
\usepackage{amssymb}
\usepackage{mathrsfs}
\usepackage{graphicx}
\usepackage{geometry}
\usepackage{diagrams}
\usepackage[hyperfootnotes=false]{hyperref}
\hypersetup{colorlinks=false, linkcolor=green, citecolor=red, hyperfootnotes=true}

\newtheorem{theorem}{Theorem}[section]

\newtheorem{corollary}[theorem]{Corollary}

\newtheorem{definition}[theorem]{Definition}
\newtheorem{example}[theorem]{Example}

\newtheorem{lemma}[theorem]{Lemma}

\newtheorem{proposition}[theorem]{Proposition}
\newtheorem{remark}[theorem]{Remark}

\begin{document}

\renewcommand{\thefootnote}{}





\makeatletter 
\newcommand*{\defeq}{\mathrel{\rlap{%
                     \raisebox{0.3ex}{$\m@th\cdot$}}%
                     \raisebox{-0.3ex}{$\m@th\cdot$}}%
                     =} 
\makeatother 

\baselineskip=13pt


\title{The Revised and Uniform Fundamental Groups \\
 and Universal Covers of Geodesic Spaces}


\author{Jay Wilkins \\
University of Connecticut \\
Department of Mathematics \\
196 Auditorium Road, Unit 3009 \\
Storrs, CT 06269-3009 \\
Email: leonard.wilkins@uconn.edu}
\address{Jay Wilkins \\ Department of Mathematics, University of Connecticut, Storrs, CT 06029}
\email{leonard.wilkins@uconn.edu}







\begin{abstract}
Sormani and Wei proved in 2004 that a compact geodesic space has a categorical universal cover if and only if its covering/critical spectrum is finite.  We add to this several equivalent conditions pertaining to the geometry and topology of the revised and uniform fundamental groups.  We show that a compact geodesic space $X$ has a universal cover if and only if the following hold: 1) its revised and uniform fundamental groups are finitely presented, or, more generally, countable; 2) its revised fundamental group is discrete as a quotient of the topological fundamental group $\pi_1^{top}(X)$. In the process, we classify the topological singularities in $X$, and we show that the above conditions imply closed liftings of all sufficiently small path loops to all covers of $X$, generalizing the traditional semilocally simply connected property.  A geodesic space $X$ with this new property is called semilocally $r$-simply connected, and $X$ has a universal cover if and only if it satisfies this condition.  We then introduce a topology on $\pi_1(X)$ called the covering topology, which always makes $\pi_1(X)$ a topological group.  We establish several connections between properties of the covering topology, the existence of simply connected and universal covers, and geometries on the fundamental group.

\emph{Key words and phrases}: geodesic space; critical spectrum; universal cover; revised fundamental group; uniform fundamental group; topological fundamental group.

2010 \emph{Mathematics Subject Classification}: Primary 54E45; Secondary 57M10, 54H11, 20F38.
\end{abstract}

\maketitle


\tableofcontents

\pagestyle{myheadings}
\markboth{J. Wilkins}{The Revised and Uniform Fundamental Groups and Universal Covers}


\section{Introduction and Main Results}

In \cite{SW2}, Sormani and Wei formally defined the covering spectrum of a compact geodesic space, a geometric invariant that detects one-dimensional holes of positive intrinsic diameter.  They showed (Theorem 3.4, \cite{SW2}) that a compact geodesic space $X$ has a universal cover if and only if its covering spectrum, $CovSpec(X)$, is finite.  When this holds, they defined the revised fundamental group of $X$ to be the deck group of the universal cover, and they showed that it is finitely generated (Proposition 6.4, \cite{SW2}).

In this paper, we extend the above results through an investigation of the geometry and topology of a slightly generalized revised fundamental group and another associated group called the uniform fundamental group.  To do so, we apply the generalized covering methods developed by Berestovskii-Plaut for uniform spaces (\cite{BPUU}) to the more restricted but important class of geodesic spaces.

In \cite{BPUU}, Berestovskii and Plaut defined the uniform universal covering and its deck group, the uniform fundamental group.  These are generalizations of the classical universal cover and fundamental group for uniform spaces - hence, metric spaces - that are not necessarily semilocally simply connected or even locally path connected.  Spaces for which the uniform universal cover exists are called coverable, and these include all geodesic spaces and, thus, Gromov-Hausdorff limits of Riemannian manifolds.  The foundation for \cite{BPUU} is discrete homotopy theory, an analog of classical path homotopy theory that uses discrete chains and chain homotopies instead of the continuous counterparts.  In \cite{W2}, and with Plaut et al. in \cite{PWetal}, the author used discrete homotopy theory to generalize the covering spectrum.  When the methods of \cite{BPUU} are applied to a metric space $X$, one obtains the $\mathbb R_+$-parameterized collection of $\varepsilon$-covers of $X$, $\{X_{\varepsilon}\}_{\varepsilon > 0}$.  These covers, in turn, determine the critical spectrum of $X$, the set of values, $Cr(X)$, at which the equivalence class of $X_{\varepsilon}$ changes as $\varepsilon$ decreases to $0$.  The uniform universal cover and uniform fundamental group are inverse limits of the $\varepsilon$-covers and their deck groups, respectively (see Section 2).

With the exception of the inverse limit formulations, this construction and spectral definition parallel those of Sormani-Wei in \cite{SW1} and \cite{SW2}.  The primary difference between the covering and critical spectra is the applicability.  The Sormani-Wei construction relies on a classical method of Spanier (\cite{S}) that requires local path connectivity of the underlying metric space $X$, which - if $X$ is compact and connected - is equivalent to being geodesic.  The Berestovskii-Plaut construction, however, can be carried out much more generally, allowing investigation of the critical spectra of more exotic and pathological metric spaces.  Like the covering spectrum, the critical spectrum detects fundamental group generators, but it also detects other metric structures in the general case that do not show up in geodesic spaces (cf. \cite{PWetal}).  Nevertheless, Plaut and the author showed in \cite{PW2} that when the underlying metric space is compact geodesic, the two spectra differ only by a constant multiple, namely $3Cr(X) = 2CovSpec(X)$.  Thus, the covering spectrum, appropriately rescaled, is a special case of the critical spectrum in the compact geodesic setting.  In particular, this fact and Sormani-Wei's theorem, together, show that a compact geodesic space has a universal cover if and only if its critical spectrum is finite.

Since we will be exploiting the methods of Berestovskii-Plaut and the uniform structure of the given geodesic space, our results will be presented in the language of discrete homotopy theory and the critical spectrum.  The relevant technical background is given in Section 2.  In this paper, a cover or covering space of $X$ will always imply a traditional, connected cover $f:Y \rightarrow X$ with the property that each $x \in X$ is contained in an evenly covered neighborhood with respect to $f$.  A \textit{universal cover} of $X$ will mean a traditional, categorical universal cover (not necessarily simply connected), or a cover $f:Y \rightarrow X$ so that, for any other cover $g:Z \rightarrow X$, there is a cover $h: Y \rightarrow Z$ such that $g \circ h = f$.  Except for the uniform universal cover, we will not need or use any of the recent, non-traditional generalizations of universal covers that relax the evenly covered property (cf.  \cite{Biss}, \cite{BS}, \cite{FZ}, \cite{PTM}).  When we use the uniform universal cover, it will always be explicitly referenced as such, so no confusion should result.

The fundamental observation that makes our results possible is that we can characterize local topology of a compact geodesic space, $X$, in terms of how path loops at a base point $* \in X$ lift not to a single $\varepsilon$-cover, $X_{\varepsilon}$, but to the $\varepsilon$-covers in the aggregate.  Thus, we begin Section 3 by slightly generalizing the revised fundamental group defined by Sormani-Wei in \cite{SW2}. The normal covering groups of the $\varepsilon$-covers, denoted by $K_{\varepsilon}$, intersect to form the \textit{closed lifting group}, the normal subgroup $\pi_{cl}(X) \unlhd \pi_1(X)$ representing all loops at $*$ that lift closed - that is, the lift is also a loop - to $X_{\varepsilon}$ for every $\varepsilon$.  The revised fundamental group, then, is $\bar{\pi}_1(X) \defeq \pi_1(X)/\pi_{cl}(X)$, and it isomorphically injects into the uniform fundamental group of $X$, denoted by $\Delta(X)$.  In fact, $\Delta(X)$ is isomorphic to $\breve{\pi}_1(X)$, the first shape group of $X$, showing that $\bar{\pi}_1(X)$ injects into $\breve{\pi}_1(X)$, as well.  When $X$ has a universal cover, $\bar{\pi}_1(X)$ agrees with the definition of Sormani-Wei, though they only define and discuss this group in that particular case.  Our approach shows that $\pi_{cl}(X)$ and $\bar{\pi}_1(X)$ are well-defined whether $X$ has a universal cover or not.  Indeed, it is a specific property of $\pi_{cl}(X)$ that determines when $X$ has a universal cover (Lemma \ref{univ cover first eq}).

Two obvious cases of interest occur when $\pi_{cl}(X)$ is trivial and when it is the whole fundamental group.  In the former case, $\bar{\pi}_1(X)$ is just the fundamental group, which, then, injects into $\Delta(X)$.  We prove (Lemma \ref{k0 trivial 1d}) that $\pi_{cl}(X)$ is always trivial for compact, one-dimensional geodesic spaces, and determining conditions under which $\pi_{cl}(X)$ is trivial in general is one of the goals of the final section (see Theorem \ref{sgtop thm2} below).  When $\pi_{cl}(X) = \pi_1(X)$, $X$ is its own universal cover and $Cr(X) = \emptyset$.  We show that these conditions are actually all equivalent (Corollary \ref{full clgroup}).

We define $X$ to be \textit{semilocally $r$-simply connected} if and only if each $x \in X$ has a neighborhood $U$ such that every path loop in $U$ based at $x$ lifts closed to $X_{\varepsilon}$ for all $\varepsilon > 0$ (Definition \ref{semi rconn}).  This generalization of the classical semilocally simply connected definition can be algebraically reformulated in a familiar way.  If $\bar{h}:\pi_1(X,x) \rightarrow \bar{\pi}_1(X,x)$ is the quotient map and $i:U \hookrightarrow X$ is the inclusion of a set $U\subset X$, then $X$ is semilocally $r$-simply connected at $x$ if and only if there is a neighborhood $U$ of $x$ such that the homomorphism $\bar{h} \circ i_{\ast}:\pi_1(U,x) \rightarrow \bar{\pi}_1(X,x)$ is trivial (Lemma \ref{equiv rconn}). There are other classical fundamental group results that have analogous statements in the revised case.  The classical functor $f \mapsto f_{\ast}$ has an analog for revised fundamental groups, and a homotopy equivalence $f:X \rightarrow Y$ induces a revised fundamental group isomorphism $f_{\sharp}:\bar{\pi}_1(X) \rightarrow \bar{\pi}_1(Y)$ (Lemma \ref{induced hom} and Corollary \ref{homotopy equiv}).

Proposition \ref{local lift property} gives sufficient conditions for $X$ to be semilocally $r$-simply connected, namely that $\bar{\pi}_1(X)$ be countable.  We then use lifting properties to classify the two types of topological singularities that obstruct semilocal simply connectedness (Definition \ref{regular singular point}).  A \textit{sequentially singular point} is one at which there is a sequence of path loops $\gamma_n$ with an associated, strictly decreasing sequence $r_n \searrow 0$ such that $\gamma_n$ lifts closed to $X_{r_n}$ but open to $X_{r_{n+1}}$.  These are ``Hawaiian earring-type'' topological singularities.  A point $x$ is \textit{degenerate} if every neighborhood of $x$ contains a nontrivial path loop that lifts closed to $X_{\varepsilon}$ for all $\varepsilon$; this is a generalization of ``not homotopically Hausdorff.'' Our first and primary theorem is

\begin{theorem}
\label{main theorem 1}
If $X$ is a compact geodesic space, then the following are equivalent.
\begin{enumerate}
\item[1)] $X$ has a universal cover.
\vspace{.025 in}
\item[2)] $Cr(X) = \frac{2}{3}CovSpec(X)$ is finite.
\vspace{.025 in}
\item[3)] The revised fundamental group, $\bar{\pi}_1(X)$, is any one of the following: \textbf{\textit{i)}} countable; \textbf{\textit{ii)}} finitely generated; \textbf{\textit{iii)}} finitely presented.
\vspace{.025 in}
\item[4)] The uniform fundamental group, $\Delta(X)$, is any one of the following: \textbf{\textit{i)}} countable; \textbf{\textit{ii)}} finitely generated; \textbf{\textit{iii)}} finitely presented.
\vspace{.025 in}
\item[5)] $X$ has no sequentially singular points.
\vspace{.025 in}
\item[6)] $X$ is semilocally $r$-simply connected.
\end{enumerate}
\noindent If these hold, then the universal cover $\hat{X}$ is $r$-simply connected \emph{(}i.e. $\bar{\pi}_1(\hat{X})$ is trivial\emph{)}, its deck and covering groups, respectively, are $\bar{\pi}_1(X)$ and $\pi_{cl}(X)$, and $\bar{\pi}_1(X)$ is isomorphic to $\Delta(X)$.
\end{theorem}

We have already noted that $1 \Leftrightarrow 2$ is known.  The proof of Theorem \ref{main theorem 1} will mostly show that the other statements are equivalent to $2$, but we include $1$ for both emphasis and reference.  Moreover, the implication $2 \Rightarrow 3iii$ follows directly from Lemma \ref{univ cover first eq} and a result of Plaut and the author in \cite{PW1}, which is recalled in Section 2.  The implications $3iii \Rightarrow 3ii \Rightarrow 3i$ are clear, and likewise for the parts of $4$.  The new and presently most important aspects of Theorem \ref{main theorem 1} are the sufficiency of $3i$ for $2$ to hold, the equivalence of $3$ - $6$, and the equivalence of $4$ - $6$ to $2$.  

The reader may want to compare Theorem \ref{main theorem 1} to Corollary 5.7 of \cite{CC}, where Cannon and Conner showed that a locally path connected, homotopically Hausdorff Peano continuum $X$ - or, equivalently, a homotopically Hausdorff compact geodesic space - has a simply connected cover if and only if $\pi_1(X)$ is any of the following: 1) countable, 2) finitely generated, 3) finitely presented.  We recharacterize Cannon-Conner's result via the critical spectrum, showing that $X$ has a simply connected cover if and only if $Cr(X)$ is finite \textit{and} $\pi_{cl}(X)$ is trivial (Proposition \ref{charac simply conn covers}).

In Section 5, we examine the relationship between universal covers and topologies on the fundamental group.  Recall that in \cite{Biss} Biss defined the \textit{topological fundamental group}, $\pi_1^{top}(X)$, to be $\pi_1(X)$ topologized as a quotient of the pointed loop space $\{f:[0,1] \rightarrow X \ | \ f(0) = f(1) = *\}$ with the compact-open topology, which is equivalent to the uniform topology when $X$ is geodesic.  Unfortunately, $\pi_1^{top}(X)$ is not always a topological group as was originally claimed in \cite{Biss}, even in the compact geodesic case; Fabel has shown that multiplication in the Hawaiian earring group with this topology is not continuous (\cite{Fabel2}).  However, Brazas has shown that $\pi_1^{top}(X)$ is a \textit{quasitopological group} (Lemma 1.8, \cite{BRAZ}), meaning that the inverse operation is continuous and multiplication is continuous in each variable, i.e. left and right translations are homeomorphisms (cf. \cite{AATM} for details on such groups).  Nevertheless, Biss' definition has still been effectively utilized.  For instance, Fabel also showed that a path connected, locally path connected metric space has a simply connected cover if and only if $\pi_1^{top}(X)$ is discrete (\cite{Fabel}).

The revised fundamental group inherits a natural quotient topology from $\pi_1^{top}(X)$, and we call the resulting group the \textit{topological revised fundamental group}, denoting it $\bar{\pi}_1^{top}(X)$.

\begin{theorem}
\label{uc cover eq discrete group}
If $X$ is compact geodesic, then $\bar{\pi}_1^{top}(X)$ is a $T_1$ quasitopological group, and $X$ has a universal cover if and only if $\bar{\pi}_1^{top}(X)$ is discrete.
\end{theorem}

We then introduce a new topology on $\pi_1(X)$ that is an example of a \textit{subgroup topology} as defined by Bogley and Sierdaski in their preprint \cite{BS}.  Specifically, the cosets of the $\varepsilon$-covering groups, $\{gK_{\varepsilon}: g \in \pi_1(X)\}$, form a basis for a topology we call the \textit{covering topology} on $\pi_1(X)$.  We denote the fundamental group with this topology by $\pi_1^{\mathcal C}(X)$, and we show that $\pi_1^{\mathcal C}(X)$ is always a topological group.  The revised fundamental group with the inherited quotient topology is denoted $\bar{\pi}_1^{\mathcal C}(X)$, and we call it the \textit{revised $\mathcal C$-fundamental group}.  This yields our final two theorems. Theorem \ref{sgtop thm1} shows that $\pi_1^{\mathcal C}(X)$ is just as effective as $\pi_1^{top}(X)$ with regard to detecting universal and simply connected covers.  Theorem \ref{sgtop thm2} gives conditions for $\pi_{cl}(X)$ to be trivial, and it provides a nice picture of the geometric connection between $\pi_1^{\mathcal C}(X)$ and $\bar{\pi}_1^{\mathcal C}(X)$.  Namely, $\bar{\pi}_1^{\mathcal C}(X)$ is in bijective correspondence with the connected components of $\pi_1^{\mathcal C}(X)$.

\begin{theorem}
\label{sgtop thm1}
If $X$ is a compact geodesic space, then $X$ has a universal cover \emph{(}respectively, simply connected cover\emph{)} if and only if $\bar{\pi}_1^{\mathcal C}(X)$ \emph{(}respectively, $\pi_1^{\mathcal C}(X)$\emph{)} is discrete.
\end{theorem}

\begin{theorem}
\label{sgtop thm2}
Let $X$ be a compact geodesic space.  Then the connected component of $\pi_1^{\mathcal C}(X)$ containing $g \in \pi_1^{\mathcal C}(X)$ is $g\pi_{cl}(X)$, and the following are equivalent.
\begin{enumerate}
\item[1)] $\pi_{cl}(X)$ is trivial.
\item[2)] $\pi_1^{\mathcal C}(X)$ is totally disconnected.
\item[3)] $\pi_1^{\mathcal C}(X)$ is Hausdorff.
\item[4)] $\pi_1^{\mathcal C}(X)$ is a geometric group over the semigroup $\mathbb R^{max}$, with geometry $\{K_{\varepsilon}\}_{\varepsilon > 0}$.
\item[5)] $\pi_1^{\mathcal C}(X)$ admits a compatible, left-invariant ultrametric.
\end{enumerate}
If these conditions hold, then $\pi_1(X)$ isomorphically injects into the first shape group of $X$, $\breve{\pi}_1(X)$.
\end{theorem}

\noindent The first statement of Theorem \ref{sgtop thm2} and the equivalence of $2$ and $3$ are part of a more general result on subgroup topologies proved by Bogley and Sierdaski in \cite{BS} (see Lemma \ref{BS lemma}).  The particular application, the last statement, and the equivalence of $1$, $4$, and $5$ to $2$ and $3$ are the new results here.

Conditions 2, 4, and 5 above are closely related.  A geometric group over an abelian, partially ordered semigroup, $S$, is a topological group, $G$, with a local basis $\{U_s\}_{s \in S}$ at the identity, $e$, such that the following hold for all $s$, $t \in S$: 1) $U_s \subset U_t$ if and only if $s \leq t$, and $\bigcup_{s \in S} U_s = G$; 2) $\bigcap_{s \in S} U_s = \{e\}$; 3) $U_s^{-1} = U_s$ and $U_s U_t = U_{s+t}$.  The collection $\{U_s\}$ is called a \textit{geometry} on $G$.  This notion was first defined by Berestovskii, Plaut, and Stallman in \cite{BPGG}, and the type of geometry that $G$ admits, if any, is strongly related to what types of metrics induce its given topology.  We show that the collection $\{K_{\varepsilon}\}_{\varepsilon > 0}$, indexed over $S=\mathbb R^{max}$ - the positive reals with their usual order but operation $a+b \defeq \max\{a,b\}$ - is always \textit{almost} a geometry on $\pi_1^{\mathcal C}(X)$, possibly lacking only one of the required conditions.  That condition is precisely part 2 of the definition, i.e. that $\pi_{cl}(X)=\bigcap_{\varepsilon > 0}K_{\varepsilon}$ is trivial.  Moreover, geometries over $\mathbb R^{max}$ correspond to ultrametrics (see Section 5), and ultrametric spaces are necessarily totally disconnected.

\vspace{.1 in}
\section{Background: Discrete Homotopy Theory}
\label{appendix}

Recall that a metric space $X$ is a \textit{geodesic space} if any two points in $X$ are joined by a (minimal) \textit{geodesic}, or an arclength parameterized curve, $\gamma:[a,b] \rightarrow X$, having length equal to the distance between its endpoints.  This is slightly different from the Riemannian definition of a geodesic, which only requires that a curve $\gamma$ be locally minimizing. It is well-known that when $X$ is geodesic, $Y$ is a connected topological space, and $f:Y\rightarrow X$ is a covering map, the metric on $X$ can be lifted to a unique geodesic metric on $Y$ that makes $f$ a local isometry.  If $X$ is also compact, this lifted metric makes $f$ a \textit{metric covering map} - a traditional covering map that is also a \textit{uniform} local isometry.  Thus, there is no distinction between metric coverings and general connected coverings in the compact geodesic setting.  

All of the spaces we consider in Section 3 and beyond will be geodesic, and compactness will be assumed when necessary.  It should be noted, however, that the Bing-Moise Theorem (Theorem 8 in \cite{Bing}, Theorem 4 in \cite{Moise}) establishes for a compact, connected, metric space the equivalence of local connectedness, local path connectedness, and the existence of a compatible geodesic metric.  Thus, every \textit{Peano continuum} - or compact, connected, locally connected metric space - admits a compatible geodesic metric, and any topological result that holds for compact geodesic spaces holds for all Peano continua.

We will outline the discrete homotopy constructions of \cite{BPUU} as applied to metric spaces.  Readers familiar with the results and methods of discrete homotopy theory may want to skip this section and simply refer back to it as needed.  Further explanations, details, and proofs may be found in \cite{BPUU} in the context of uniform spaces, and in \cite{W2} and \cite{PWetal} for metric spaces.

Let $X$ be a connected metric space, and fix $\varepsilon >0$. An $\varepsilon $\textit{-chain} $\alpha $ in $X$ is a finite sequence $\alpha =\{x_{0},x_{1},\dots ,x_{n}\}$ such that $d(x_{i-1},x_{i})<\varepsilon $ for $i=1,\dots ,n$. A \textit{basic move} on an $\varepsilon $-chain is the addition or removal of a single point with the conditions that the endpoints remain fixed and the resulting chain is still an $\varepsilon $-chain. Two $\varepsilon $-chains $\alpha $ and $\beta $ are $\varepsilon $\textit{-homotopic} if there is a finite sequence of $\varepsilon$-chains, $H=\{\alpha =\gamma _{0},\gamma _{1},\dots ,\gamma _{k-1},\gamma _{k}=\beta \}$ - called an \textit{$\varepsilon $-homotopy} - such that each $\gamma _{i}$ differs from $\gamma _{i-1}$ by a basic move. The relation ``$\varepsilon$-homotopic'' is an equivalence relation on $\varepsilon$-chains in $X$, and it carries the same basic concatenation and algebraic or groupoid properties as traditional path homotopy equivalence.  

For a fixed base point $\ast \in X$, $X_{\varepsilon }$ is the set of all equivalence classes of $\varepsilon $-chains in $X$ beginning at $\ast$, $[\{\ast =x_{0},\dots,x_{n}\}]_{\varepsilon }$, and $\varphi _{\varepsilon }:X_{\varepsilon }\rightarrow X$ is the endpoint map taking $[\{\ast=x_{0},\dots ,x_{n}\}]_{\varepsilon }$ to $x_{n}$. There is a natural metric, $d_{\varepsilon}$, on $X_{\varepsilon}$ that makes $\varphi_{\varepsilon}$ a regular metric covering map (cf. \cite{PWetal}, \cite{PW1}, or \cite{W2}). We call $X_{\varepsilon}$ and its deck group, $\pi_{\varepsilon}(X)$, the \textit{$\varepsilon$-cover} and $\varepsilon$\textit{-group} of $X$, respectively.  The $\varepsilon$-group is naturally identified with the subset of $X_{\varepsilon}$ consisting of classes of $\varepsilon$\textit{-loops} at $*$, which is a group under concatenation.  If an $\varepsilon $-loop $\alpha =\{\ast=x_{0},\dots ,x_{n}=\ast \}$ is $\varepsilon $-homotopic to the trivial loop $\{\ast \}$ then we say $\alpha $ is $\varepsilon $-\textit{null}.  Clearly, the identity of $\pi_{\varepsilon}(X)$ is $[\{*\}]_{\varepsilon}$, and we take $\varphi_{\varepsilon}:(X_{\varepsilon},\tilde{*}) \rightarrow (X,*)$ to be a pointed map, where $\tilde{*}$ will always be a shorthand notation for the identity element $[\{*\}]_{\varepsilon}\in \pi_{\varepsilon}(X) \subset X_{\varepsilon}$.

The $\varepsilon$-group is uniformly discrete and left invariant as a metric subspace of $X_{\varepsilon}$.  It acts discretely by isometries on $X_{\varepsilon}$ via left concatenation. That is, $[\alpha]_{\varepsilon}([\beta]_{\varepsilon}) = [\alpha \beta]_{\varepsilon}$ for $[\alpha]_{\varepsilon} \in \pi_{\varepsilon}(X)$ and $[\beta]_{\varepsilon} \in X_{\varepsilon}$, and if $d_{\varepsilon}([\alpha]_{\varepsilon}[\beta]_{\varepsilon},[\beta]_{\varepsilon}) < \varepsilon$, then $[\alpha]_{\varepsilon}$ is trivial.  A discrete action is necessarily free and properly discontinuous.  The resulting metric quotient $X_{\varepsilon}/\pi_{\varepsilon}(X)$ is homeomorphic and uniformly locally isometric to $X$; the two are isometric when $X$ is geodesic.  When $X$ is compact geodesic, the $\varepsilon$-groups are finitely presented (Theorem 3, \cite{PW1}), with a set of generators $\{[\gamma_i]_{\varepsilon}\}_{i=1}^n$ and relations of the form $[\gamma_i]_{\varepsilon}[\gamma_j]_{\varepsilon}=[\gamma_k]_{\varepsilon}$.

The preceding construction is independent of the base point. If $*'$ is another base point and $\alpha$ is a fixed $\varepsilon$-chain from $*'$ to $*$, then the maps $[\beta]_{\varepsilon} \mapsto [\alpha \beta]_{\varepsilon}$ and $[\gamma]_{\varepsilon} \mapsto [\alpha \gamma \alpha^{-1}]_{\varepsilon}$ are, respectively, an isometric covering equivalence from $(X_{\varepsilon},\tilde{*})$ to $(X_{\varepsilon},\tilde{*}')$ and an isomorphism from $\pi_{\varepsilon}(X,*)$ to $\pi_{\varepsilon}(X,*')$. Thus, we usually just fix a base point $*$ in $X$ and use it to determine all $\varepsilon$-covers and groups.  This gives us collections $\{X_{\varepsilon}\}_{\varepsilon > 0}$ and $\{\pi_{\varepsilon}(X)\}_{\varepsilon > 0}$, which stratify the covering spaces and fundamental group of $X$.  Indeed, if $X$ is compact geodesic, every connected cover of $X$ is covered by $X_{\varepsilon}$ for small enough $\varepsilon$.

Now, we assume that $X$ is geodesic.  The natural metric on $X_{\varepsilon}$ mentioned above is equal to the lifted geodesic metric from $X$ (Proposition 24, \cite{PW1}).  Given $0<\delta <\varepsilon $, there is a well-defined, surjective bonding map $\varphi _{\varepsilon \delta}:X_{\delta }\rightarrow X_{\varepsilon }$ that simply treats a $\delta $-chain as an $\varepsilon $-chain, i.e. $\varphi _{\varepsilon \delta}([\alpha ]_{\delta })=[\alpha ]_{\varepsilon }$.  These maps are metric covering maps between geodesic spaces and, thus, isometries when they are injective.  They also satisfy the composition relation $\varphi_{\varepsilon \tau} = \varphi_{\varepsilon \delta} \circ \varphi_{\delta \tau}$ when $0 < \tau < \delta < \varepsilon$. The restriction of $\varphi_{\varepsilon \delta}$ to $\pi_{\delta}(X) \subset X_{\delta}$ is a homomorphism onto $\pi_{\varepsilon}(X)\subset X_{\varepsilon}$, which we denote by $\Phi_{\varepsilon \delta}:\pi_{\delta}(X) \rightarrow \pi_{\varepsilon}(X)$. For any $0 < \delta < \varepsilon$, $\Phi_{\varepsilon \delta}$ is injective if and only if $\varphi_{\varepsilon \delta}$ is injective, and these homomorphisms satisfy the obvious analog of the previously mentioned composition relation.

A positive number $\varepsilon$ is a \textit{critical value} of a compact geodesic space $X$ if there is a nontrivial $\varepsilon$-loop, $\gamma$, at $*$ that is $\delta$-null for all $\delta > \varepsilon$.  Equivalently, there is a nontrivial element $[\gamma]_{\varepsilon} \in \pi_{\varepsilon}(X)$ that is in $ker \ \Phi_{\delta \varepsilon}$ for all $\delta > \varepsilon$.    The set $Cr(X)$ of all critical values of $X$ is called the \textit{critical spectrum} of $X$. \textit{It should be noted that our present definition of a critical value relies on the assumption that $X$ is compact geodesic.}  As we mentioned in the introduction, critical values of metric spaces can be defined more generally, and the definition given here would not suffice to capture every value that should be critical for a non-compact or non-geodesic space.  See \cite{PWetal} for the general definition and a classification of the types of critical values.  For compact geodesic $X$, however, critical values occur only in this very specific way (Lemma 3.1.10, \cite{W2}), allowing for the present simpler definition.

The critical/covering spectrum of a compact geodesic space is closed, discrete, and bounded above by $diam(X)$ in $\mathbb R_+=(0,\infty)$, although $\inf Cr(X)$ may be $0$.  Thus, either $Cr(X)$ is finite, or it is a strictly decreasing sequence of isolated critical values converging to $0$. This was originally shown by Sormani-Wei for $CovSpec(X)$ in \cite{SW2}, and Plaut-Wilkins in \cite{PW1} constructed a different direct proof for $Cr(X)$ not using the equality $3Cr(X) = 2CovSpec(X)$.  We have noted that these spectra detect fundamental group generators.  The standard examples illustrating this involve circles (and not by coincidence - see \cite{PW1}).  For instance, if $X$ is the geodesic circle of circumference $r$, $X = S_r^1$, then $X_{\varepsilon} = X$ and $\pi_{\varepsilon}(X)$ is trivial for $\varepsilon > \frac{r}{3}$, while $X_{\varepsilon} = \mathbb R$ and $\pi_{\varepsilon}(X) \cong \mathbb Z$ for $0 < \varepsilon \leq \frac{r}{3}$ (Example 17, \cite{PWetal}).  Thus, $Cr(X) = \{\frac{r}{3}\}$.

There is a natural way to transfer properties between discrete chains and continuous paths.  The proofs of the statements leading up to Definition \ref{ehomomorph} may be found in \cite{W2}.

\begin{definition}
For a path $\gamma:[a,b] \rightarrow X$, a strong $\varepsilon$-chain along $\gamma$ is an $\varepsilon$-chain $\alpha = \{x_0,x_1,\dots,x_n\}$ with the following property: there exists a partition $a = t_0 < t_1 < \cdots < t_n = b$ of $[a,b]$ such that $\gamma(t_i) = x_i$ for each $i=0,1,\dots,n$ and $\gamma([t_{i-1},t_i]) \subset B_{\varepsilon}(\gamma(t_{i-1})) \cap B_{\varepsilon}(\gamma(t_i))$ for each $i=1,\dots,n$.
\end{definition}

\noindent Note that the reversal of a strong $\varepsilon$-chain along $\gamma$ is a strong $\varepsilon$-chain along $\gamma^{-1}$, and if $\alpha_1$ and $\alpha_2$ are strong $\varepsilon$-chains along paths $\gamma_1$ and $\gamma_2$, respectively, with the initial point of $\gamma_2$ equal to the terminal point of $\gamma_1$, then the concatenation $\alpha_1 \alpha_2$ is a strong $\varepsilon$-chain along $\gamma_1 \gamma_2$.  

A simple Lebesgue covering argument shows that there is a strong $\varepsilon$-chain along any path, and if $\gamma$ is any path, then any two strong $\varepsilon$-chains along $\gamma$ are $\varepsilon$-homotopic.  Moreover, if $\gamma$ and $\lambda$ are paths that are fixed endpoint path homotopic, then any strong $\varepsilon$-chain along $\gamma$ is $\varepsilon$-homotopic to any strong $\varepsilon$-chain along $\lambda$. These statements are not true for $\varepsilon$-chains along paths without the strong condition.  Taken together, these properties induce natural, well-defined, ``continuous to discrete'' homomorphisms from $\pi_1(X)$ to each $\varepsilon$-group, which are surjective when $X$ is geodesic.

\begin{definition}
\label{ehomomorph}
Fix a base point $* \in X$, and let $\pi_1(X,*)$ and $\pi_{\varepsilon}(X,*)$ be the fundamental and $\varepsilon$-groups, respectively, based at $*$.  For $\varepsilon > 0$, the $\varepsilon$-homomorphism is the map $h_{\varepsilon}:\pi_1(X,*) \rightarrow \pi_{\varepsilon}(X,*)$ taking $[\gamma] \in \pi_1(X,*)$ to the $\varepsilon$-equivalence class of strong $\varepsilon$-loops along $\gamma$.
\end{definition}

\noindent If $X$ is geodesic and $\alpha$ is an $\varepsilon$-loop at $*$, we can join each consecutive pair of points in $\alpha$ by a minimal geodesic, making $\alpha$ a strong $\varepsilon$-chain along the resulting broken geodesic path loop.  This is an \textit{$\varepsilon$-chording} of $\alpha$, and the homotopy class of the resulting path loop maps to $[\alpha]_{\varepsilon}$ under $h_{\varepsilon}$, making $h_{\varepsilon}$ surjective.  It is easy to see that the following commutes, where $\iota$ is the identity isomorphism.
\begin{equation}
\label{comm diagram 1}
\begin{diagram}
\pi_{\delta}(X,*) & \rTo_{\Phi_{\varepsilon \delta}} & \pi_{\varepsilon}(X,*) \\
\uTo^{h_{\delta}} & & \uTo_{h_{\varepsilon}} &  \\
\pi_1(X,*) & \rTo^{\iota} & \pi_1(X,*)
\end{diagram}
\end{equation}

The following definition and lemma (Definition 16 and Proposition 17 in \cite{PW1}) will be needed for some basic chain homotopy computations later on.
\begin{definition}
\label{chain gap def}
Let $X$ be a metric space and $\varepsilon > 0$.  Given an $\varepsilon$ chain in $X$, $\alpha=\{x_0,x_1,\dots,x_n\}$, define $E(\alpha)\defeq \min_{1 \leq i \leq n} \{\varepsilon - d(x_i,x_{i+1})\}$.  If $\alpha=\{x_0,x_1,\dots,x_n\}$ and $\beta=\{y_0,y_1,\dots,y_n\}$ are chains having the same number of points, define $\mathcal D(\alpha,\beta) \defeq \max_{0 \leq i \leq n} \{d(x_i,y_i)\}$.
\end{definition}

\begin{lemma}
\label{close chain lemma}
Let $\alpha$ be an $\varepsilon$-chain in a metric space $X$.  If $\beta$ is a chain with the same endpoints and same number of points as $\alpha$, and if $\mathcal D(\alpha,\beta) < \frac{1}{2}E(\alpha)$, then $\beta$ is an $\varepsilon$-chain that is $\varepsilon$-homotopic to $\alpha$.
\end{lemma}

The relations $\varphi_{\varepsilon \tau} = \varphi_{\varepsilon \delta}  \circ  \varphi_{\delta \tau}$ and $\Phi_{\varepsilon \tau} = \Phi_{\varepsilon \delta}  \circ  \Phi_{\delta \tau}$ for $\tau < \delta < \varepsilon$ imply that $\{X_{\varepsilon},\varphi_{\varepsilon \delta}\}$ and $\{\pi_{\varepsilon}(X),\Phi_{\varepsilon \delta}\}$ form inverse systems indexed by $\mathbb R_+$ with reverse order.  The \textit{uniform universal cover of $X$} (the \textit{UU-cover} for short) and \textit{uniform fundamental group of $X$} are the resultant inverse limits
\[\tilde{X}=\lim_{\leftarrow} X_{\varepsilon} \quad \text{and} \quad \Delta(X) = \lim_{\leftarrow} \pi_{\varepsilon}(X).\]
\noindent  The endpoint projection $\tilde{\varphi}:\tilde{X} \rightarrow X$ is surjective and continuous but is not typically a traditional cover; the fibers $\tilde{\varphi}^{-1}(x)$, for instance, are totally disconnected but not necessarily discrete.  However, $\tilde{X}$ is a generalized universal cover in the following senses: 1) (\textit{universality}) if $f:Y \rightarrow X$ is a cover then there is a unique, possibly generalized, cover $\tilde{f}:\tilde{X} \rightarrow Y$ such that $f \circ \tilde{f} = \tilde{\varphi}$; \ 2) (\textit{unique lifting}) paths and path homotopies lift uniquely into $\tilde{X}$; 3) (\textit{generalized regularity}) $\Delta(X)$ - which admits equivalent definitions as $\tilde{\varphi}^{-1}(*)$ and as the deck group of $\tilde{\varphi}:\tilde{X} \rightarrow X$ - acts prodiscretely on $\tilde{X}$, and $\tilde{X}/\Delta(X)$ is homeomorphic to $X$.  Thus, $\Delta(X)$ can be interpreted as a generalized fundamental group of $X$.

The following facts are in \cite{BPUU}. There is a canonical homomorphism $\Lambda:\pi_1(X) \rightarrow \Delta(X)$, mapping $[\gamma] \in \pi_1(X,*)$ to the endpoint of the lift of $\gamma$ at the identity in $\Delta(X) \subset \tilde{X}$.  While $\Lambda$ is surjective if and only if $\tilde{X}$ is path connected, $\Lambda(\pi_1(X))$ is always a dense, normal subgroup of $\Delta(X)$ when the latter is given the inverse limit topology, which is also the subspace topology it inherits from $\tilde{X}$. The kernel of $\Lambda$ contains those elements in $\pi_1(X)$ represented by loops that lift closed to $\tilde{X}$, or, equivalently, that lift closed to $X_{\varepsilon}$ for all $\varepsilon > 0$. That is, $ker \ \Lambda = \pi_{cl}(X)$; we will say more about this in Section 3. 

Now, if $X$ is compact and $Cr(X)$ is finite, then the covers $X_{\varepsilon}$ stabilize as $\varepsilon \searrow 0$.  Precisely, $\varphi_{\varepsilon \delta}:X_{\delta}\rightarrow X_{\varepsilon}$ and $\Phi_{\varepsilon \delta}:\pi_{\delta}(X) \rightarrow \pi_{\varepsilon}(X)$ are, respectively, covering equivalences and isomorphisms for all $0 < \delta < \varepsilon \leq \min Cr(X)$.  In this case, for each $0 < \varepsilon \leq \min Cr(X)$, $\tilde{X}$ is equivalent to the traditional cover $X_{\varepsilon}$ with deck group $\Delta(X) \cong \pi_{\varepsilon}(X)$.  The universality of $\tilde{X}$ thus implies that it - as well as each $X_{\varepsilon}$, $0 < \varepsilon \leq \min Cr(X)$ - is a traditional universal cover.  This is precisely the idea Sormani-Wei used to prove that $X$ has a universal cover when $CovSpec(X)=\frac{3}{2}Cr(X)$ is finite, though they did so without reference to the UU-cover or any generalized universal cover.  Of course, they also proved the converse; a universal cover requires that the spectra be finite, and the preceding scenario still holds.

\begin{remark}
We have changed notation slightly from \cite{BPUU}.  In \emph{\cite{BPUU}}, $\Lambda$ and the uniform fundamental group are denoted, respectively, by $\lambda$ and $\delta_1(X)$, but we will use $\lambda$ and $\delta$ differently.
\end{remark}

Finally, when $X$ is compact geodesic, Brodskiy, Dydak, Labuz, and Mitra showed (Corollary 6.5, \cite{BD1}) that the uniform fundamental group is isomorphic to \textit{the first shape group of $X$}, $\breve{\pi}_1(X)$.  We will suppress the formal and rather technical definition of $\breve{\pi}_1(X)$ (cf. \cite{Dydak}), since it will play no role in this paper.  Like $\Delta(X) = \lim_{\leftarrow} \pi_{\varepsilon}(X)$, $\breve{\pi}_1(X)$ is an inverse limit of coarse approximations to the fundamental group.  Very roughly, if $|N(\mathcal U)|$ is the geometric realization of the nerve of an open covering, $\mathcal U$, and if $\mathcal V$ is a refinement of $\mathcal U$, then there is a bonding homomorphism $\rho_{\mathcal U \mathcal V}:\pi_1(|N(\mathcal U)|) \rightarrow \pi_1(|N(\mathcal V)|)$.  Then $\breve{\pi}_1(X)$ is the resultant inverse limit $\lim_{\leftarrow} \pi_1(|N(\mathcal U)|)$, taken over a cofinal directed set consisting of open coverings admitting subordinated partitions of unity.  Like finding conditions under which $\pi_{cl}(X)$ is trivial in the geodesic case, determining when the fundamental group isomorphically injects into the first shape group is a broad area of interest in general topology (cf. \cite{FZ}, \cite{FZ2}).

\vspace{.1 in}
\section{The Revised Fundamental Group}

Much of this section will hold for general geodesic spaces, since many of the results will not involve our compact-dependent definition of a critical value.  Recall that open metric balls of radius $r$ and center $x$ are denoted $B_r(x)$.  To distinguish discrete chains from continuous paths, we will exclusively use `path' (resp. `path loop') to denote continuous curves (resp. closed curves).  A path loop, $\gamma:[a,b] \rightarrow X$ is said to be \textit{based at $x$} if $\gamma(a)=\gamma(b)=x$.  We do not discuss free homotopies in the present work, so to say that two path loops based at a point $x$ are homotopic means that they are fixed endpoint path homotopic, or homotopic rel $x$.  If groups $G$ and $H$ are isomorphic, we will denote this by $G\cong H$.

Let $X$ be a geodesic space with base point $*$, and let $\{X_{\varepsilon}\}$ be the $\varepsilon$-covers determined by this base point.  Recalling Definition \ref{ehomomorph}, we denote the subgroup $ker \ h_{\varepsilon} \unlhd \pi_1(X,*)$ by $K_{\varepsilon}$, or $K_{\varepsilon}(*)$ when the base point needs to be emphasized.  Commutative diagram (\ref{comm diagram 1}) shows immediately that these kernels form a decreasing, nested set of normal subgroups of $\pi_1(X,*)$ in the following sense: if $0 < \delta < \varepsilon$, then $K_{\delta} \subseteq K_{\varepsilon}$.  We thus define the normal subgroup
\[\pi_{cl}(X,*) \defeq \bigcap_{\varepsilon > 0} K_{\varepsilon} \ \unlhd \ \pi_1(X,*),\]
\noindent and, for reasons which will soon become clear, we call $\pi_{cl}(X,*)$ the \textit{closed lifting group} of $X$ at $*$.  When $\pi_{cl}(X,*)$ is the trivial subgroup of $\pi_1(X,*)$ consisting of just the identity, we denote this by $\pi_{cl}(X,*) = 1$.

\begin{definition}
\label{revised fg}
The \textit{revised fundamental group} of $X$ at $*$ is defined to be the quotient group $\bar{\pi}_1(X,*) \defeq \pi_1(X,*)/\pi_{cl}(X,*)$, and we denote the standard quotient homomorphism by $\bar{h}:\pi_1(X,*) \rightarrow \bar{\pi}_1(X,*)$.
\end{definition}

\noindent We usually suppress the base point when it is clear, and we show below that it is, in fact, immaterial.  We will routinely express elements of $\bar{\pi}_1(X)$ as left cosets, $[\gamma]\pi_{cl}(X)$.  The only scenario where confusion might occur is when we need to distinguish the subset $[\gamma]\pi_{cl}(X) \subset \pi_1(X)$ from the corresponding element in $\bar{\pi}_1(X)$.  The context, however, should always make the usage clear.

We can connect the fundamental group, revised fundamental group, and $\varepsilon$-groups by a commutative diagram.  We use a standard generalization of the First Isomorphism Theorem, namely the following: if $f:G \rightarrow H$ is a homomorphism, not necessarily surjective, and if $ker \ f$ contains a normal subgroup, say $N$, then there is a unique homomorphism $h:G/N \rightarrow H$ such that $h \circ q = f$, where $q:G \rightarrow G/N$ is the quotient homomorphism. The map $h$ is defined by $h(gN) = f(g)$.  We will use this result several times.

For any $\varepsilon > 0$, $K_{\varepsilon}$ contains $\pi_{cl}(X)$.  Thus, there is a unique homomorphism $\bar{h}_{\varepsilon}:\bar{\pi}_1(X) \rightarrow \pi_{\varepsilon}(X)$ which takes $[\gamma]\pi_{cl}(X)$ to $h_{\varepsilon}([\gamma])$ and is such that the following diagram commutes.
\begin{equation}
\label{comm diagram 2}
\begin{diagram}
\pi_1(X) & \rTo_{h_{\varepsilon}} & \pi_{\varepsilon}(X) \\
\dTo_{\bar{h}} & \ruTo^{\bar{h}_{\varepsilon}} &  \\
 \bar{\pi}_1(X) &
\end{diagram}
\end{equation}
\noindent Note that $ker \ \bar{h}_{\varepsilon}=\bar{h}(K_{\varepsilon})$, and since $h_{\varepsilon}$ is surjective, $\bar{h}_{\varepsilon}$ is surjective, also.

Let $\gamma:[a,b] \rightarrow X$ be a path in $X$ beginning at $*$, and let $\tilde{\gamma}$ denote its unique lift to the identity $\tilde{*} \in \pi_{\varepsilon}(X) \subset X_{\varepsilon}$.  It follows from Proposition 21 in \cite{PW1} that the endpoint of $\tilde{\gamma}$ is the equivalence class of strong $\varepsilon$-chains along $\gamma$.  Thus, if $\gamma$ is a path loop then the endpoint of $\tilde{\gamma}$ is $h_{\varepsilon}([\gamma])$, which is trivial if and only if $[\gamma] \in K_{\varepsilon}$. Hence, we have

\begin{lemma}
\label{charac lifts}
Let $X$ be a geodesic space, and let $\varepsilon > 0$ be given.  If $\gamma:[a,b] \rightarrow X$ is a path loop at $* \in X$, then $\tilde{\gamma}$ is a path loop at $\tilde{*} \in X_{\varepsilon}$ if and only if some \emph{(}equivalently, every\emph{)} strong $\varepsilon$-chain along $\gamma$ is $\varepsilon$-null.  That is, $\gamma$ lifts closed to $\tilde{*} \in X_{\varepsilon}$ if and only if $[\gamma] \in K_{\varepsilon}$, and $K_{\varepsilon}$ is the covering group of $X_{\varepsilon}$.
\end{lemma}

\noindent It follows from the regularity of $\varphi_{\varepsilon}:X_{\varepsilon} \rightarrow X$ that the lifts to $X_{\varepsilon}$ of a path loop based at any point in $X$ are either all closed or all open, regardless of the point in the appropriate preimage to which it is lifted.  Hence, if $\gamma$ is a path loop based at $x$, there is no ambiguity in simply stating that $\gamma$ lifts closed or open to $X_{\varepsilon}$ without specifying the particular point in $\varphi_{\varepsilon}^{-1}(x)$ to which it is lifted.

Lemma \ref{charac lifts} also gives us another useful and familiar interpretation of a critical value.  Suppose $X$ is compact and $\varepsilon > 0$ is a critical value of $X$.  Then there is a nontrivial $\varepsilon$-loop, $\alpha$, at $*$ that is $\delta$-null for all $\delta > \varepsilon$.  Let $\gamma$ be an $\varepsilon$-chording of $\alpha$, and note that $\alpha$ is also a strong $\delta$-loop along $\gamma$ for all $\delta > \varepsilon$.  It follows, then, that $\gamma$ lifts open to $X_{\varepsilon}$ since $\alpha$ is $\varepsilon$-nontrivial, but closed to $X_{\delta}$ for all $\delta > \varepsilon$ since $\alpha$ is $\delta$-null.  That is, there is a path loop at $*$ that lifts open to $X_{\varepsilon}$ but closed to $X_{\delta}$ for all $\delta > \varepsilon$.  This is precisely how one characterizes elements of the covering spectrum of Sormani-Wei, and we see that it is a consequence of the more general discrete formulation.

\begin{corollary}
\label{loop in a ball}
Let $X$ be geodesic with base point $*$.  If $\gamma$ is a path loop based at $x \in X$ and lying in an open ball of radius $\varepsilon$ not necessarily centered at $x$, then $\gamma$ lifts closed to $X_{\varepsilon}$.
\end{corollary}

\begin{proof}
Suppose $\gamma$ lies in $B_{\varepsilon}(\bar{x})$. Choose a strong $\varepsilon$-loop along $\gamma$, say $\beta = \{x=x_0,x_1,\dots,x_n=x\}$, and let $\alpha$ be the 2-point $\varepsilon$-chain $\{\bar{x},x\}$.  Then $\alpha \beta \alpha^{-1}$ is an $\varepsilon$-loop, and note that $\beta$ is $\varepsilon$-null if and only if $\alpha \beta \alpha^{-1}$ is.  But, letting $\overset{\varepsilon}\sim$ denote the relation ``$\varepsilon$-homotopic,'' we have
\[\alpha \beta \alpha^{-1} \overset{\varepsilon}\sim \{\bar{x},x=x_0,x_1,\dots,x_{n-1},x_n=x,\bar{x}\}.\]
\noindent Since $x_i \in B_{\varepsilon}(\bar{x})$ for each $i=0,1,\dots,n$, we can successively remove $x_0$, then $x_1$, and so on, reducing $\alpha \beta \alpha^{-1}$ via $\varepsilon$-homotopy to the trivial chain $\{\bar{x},\bar{x}\}$. Thus, $\beta$ is $\varepsilon$-null.

Finally, let $\lambda$ be a path from $*$ to $x$, and choose a strong $\varepsilon$-chain, $\sigma$, along $\lambda$.  Then $\sigma \beta \sigma^{-1}$ is a strong $\varepsilon$-chain along $\lambda \gamma \lambda^{-1}$ that is $\varepsilon$-null.  By Lemma \ref{charac lifts}, this means that $\lambda \gamma \lambda^{-1}$ lifts closed to $\tilde{*} \in X_{\varepsilon}$, which, by uniqueness of path lifts, can only hold if $\gamma$ lifts closed, also.\end{proof}

\vspace{.1 in}
\noindent The author has shown that the previous corollary can be strengthened to path loops lying in balls of radius $\frac{3\varepsilon}{2}$ (\cite{W2}).  The proof is more technical, however, and we will not need that stronger result.

We include the next definition and lemma for convenient reference, since they will be used several times.  The lemma is an immediate consequence of Lemma \ref{charac lifts} and Corollary 52 in \cite{PW2}.

\begin{definition}
A lollipop at $x \in X$ is a path loop of the form $\alpha \beta \alpha^{-1}$, where $\alpha$ is a path from $x$ to a point $y$ and $\beta$ is a path loop at $y$.  If $\beta$ - which we call the head of the lollipop - lies in a ball of radius $\varepsilon$, then we call $\alpha \beta \alpha^{-1}$ an $\varepsilon$-lollipop.  If a path loop $\gamma$ is homotopic to a product of $\varepsilon$-lollipops, $\alpha_1 \beta_1 \alpha_1^{-1} \cdots \alpha_m \beta_m \alpha_m^{-1}$, we call this an $\varepsilon$-lollipop factorization of $[\gamma]$.
\end{definition}

\begin{lemma}
\label{geodesic loop reduction}
Let $X$ be a geodesic space.  If $[\gamma] \in \pi_1(X,*)$ is in $K_{\varepsilon}$, then $\gamma$ is homotopic to a finite product of $\frac{3\varepsilon}{2}$-lollipops.
\end{lemma}

\noindent It follows from the uniqueness of path lifts to covering spaces that a lollipop $\alpha \beta \alpha^{-1}$ lifts closed to a regular cover if and only if $\beta$ lifts closed.  Combining this with Corollary \ref{loop in a ball}, we obtain

\begin{corollary}
\label{lollipop products}
If $\gamma = \alpha_1 \beta_1 \alpha_1^{-1} \cdots \alpha_n \beta_n \alpha_n^{-1}$ is a product of $\varepsilon$-lollipops in a geodesic space $X$, then $\gamma$ lifts closed to $X_{\varepsilon}$.
\end{corollary}  

\begin{remark}
\label{SW connection}
Lemma \emph{\ref{geodesic loop reduction}} and the statement immediately following the proof of Corollary \emph{\ref{loop in a ball}} are the means by which one shows that $\frac{3}{2}Cr(X) = CovSpec(X)$.  In fact, those results, together, show that $K_{\varepsilon} = \pi^{3\varepsilon/2}(X,*)$, where $\pi^{\delta}(X,*) \subset \pi_1(X,*)$ is the subgroup generated by $\delta$-lollipops.  These are precisely the covering groups of the $\delta$-covers, $X^{\delta}$, used by Sormani and Wei to define their covering spectrum \emph{(}Definition 2.3, \cite{SW2}\emph{)}.  With this fact, it is straightforward to show that $X_{\varepsilon}$ is isometricaly equivalent to $X^{3\varepsilon/2}$, and the equality $\frac{3}{2}Cr(X) = CovSpec(X)$ then follows immediately.  See \cite{PW2} for details.
\end{remark}

\begin{proposition}
\label{characterize k0}
If $X$ is geodesic, then $[\gamma] \in \pi_{cl}(X,*)$ if and only if $\gamma$ lifts closed to $X_{\varepsilon}$ for all $\varepsilon > 0$.  If $X$ is also compact, then $[\gamma] \in \pi_{cl}(X,*)$ if and only if $\gamma$ lifts closed to all connected covers of $X$.
\end{proposition}

\begin{proof}
The first statement follows directly from Lemma \ref{charac lifts} and the definition of $\pi_{cl}(X)$.  The second statement follows from the first and the fact that, in the compact case, every connected cover of $X$ is covered by $X_{\varepsilon}$ for some sufficiently small $\varepsilon$.\end{proof}

\vspace{.1 in}
Thus, $\pi_{cl}(X)$ is precisely equal to the kernel of $\Lambda:\pi_1(X) \rightarrow \Delta(X)$, and we now see why $\pi_{cl}(X)$ is called the closed lifting group.  It is natural to wonder why $\pi_{cl}(X)$ is not simply defined as $ker \ \Lambda$.  While this would be more efficient, the homomorphisms $h_{\varepsilon}$ and the representation $\pi_{cl}(X) = \bigcap_{\varepsilon > 0} K_{\varepsilon}$ are useful by themselves (see Section 4), beyond the mere fact that $\pi_{cl}(X) = ker \ \Lambda$.  Therefore, it will be convenient to have both the algebraic and geometric interpretations of the closed lifting group.

We have carried out the above constuctions using the $\varepsilon$-covers determined by an arbitrarily chosen base point, $*$.  Since our definitions fundamentally rely upon the lifts of path loops to these covers, this raises the question of whether or not our results depend on that base point, particularly since another base point would technically induce different, albeit isometrically equivalent, covers.  This equivalence and the regularity of the $\varepsilon$-covers should make it evident that the lifts of path loops are independent of the point we choose to construct the $\varepsilon$-covers.  For the sake of completeness, however, we will formally clear up any questions about base points.  

The key result is one we already mentioned above:  if $*$ is our base point in $X$, and $\gamma$ is a path loop based at $x$, then $\gamma$ lifts closed to $X_{\varepsilon}$ if and only if every lollipop $\alpha \gamma \alpha^{-1}$ lifts closed to $X_{\varepsilon}$, where $\alpha$ is any path from $*$ to $x$.  Recall that we denote the base point of $X_{\varepsilon}$ by $\tilde{*}$, which we always take to be the identity in $\pi_{\varepsilon}(X,*) \subset X_{\varepsilon}$.  For the next two lemmas, we adopt the following notation, but we will not need it thereafter.  We denote the $\varepsilon$-covers determined by $*$ using the standard pointed notation, $(X_{\varepsilon},\tilde{*})$.

\begin{lemma}
\label{first bp indep}
Let $X$ be a geodesic space with base points $*_1$ and $*_2$.  A path loop $\gamma$ based at $x \in X$ lifts closed to $(X_{\varepsilon},\tilde{*}_1)$ if and only if $\gamma$ lifts closed to $(X_{\varepsilon},\tilde{*}_2)$.
\end{lemma}

\begin{proof}
This follows from repeated applications of the aforementioned key result regarding lollipops.  Let $\psi:(X_{\varepsilon},\tilde{*}_1) \rightarrow (X_{\varepsilon},\tilde{*}_2)$ be an isometric covering equivalence (see Section 2) satisfying $\varphi_{\varepsilon}^{2} \circ \psi = \varphi_{\varepsilon}^{1}$, where $\varphi_{\varepsilon}^{1}:(X_{\varepsilon},\tilde{*}_1) \rightarrow (X,*_1)$ and $\varphi_{\varepsilon}^{2}:(X_{\varepsilon},\tilde{*}_2) \rightarrow (X,*_2)$ are the respective covers.  If $\gamma$ lifts closed to $(X_{\varepsilon},\tilde{*}_1)$, then $\alpha \gamma \alpha^{-1}$ lifts closed to $(X_{\varepsilon},\tilde{*}_1)$, where $\alpha$ is any path from $*_1$ to $x$.  Let $\tilde{\lambda}$ denote the lift of $\alpha \gamma \alpha^{-1}$ to $\tilde{*}_1 \in (X_{\varepsilon},\tilde{*}_1)$.  Then $\psi \circ \tilde{\lambda}$ is a path loop in $(X_{\varepsilon},\tilde{*}_2)$, and it projects to $\alpha \gamma \alpha^{-1}$ under $\varphi_{\varepsilon}^2$, because of the equality $\varphi_{\varepsilon}^{2} \circ \psi = \varphi_{\varepsilon}^{1}$.  Thus, $\alpha \gamma \alpha^{-1}$ lifts closed to $(X_{\varepsilon},\tilde{*}_2)$, which further implies that $\beta \alpha \gamma \alpha^{-1} \beta^{-1} = (\beta \alpha)\gamma(\beta \alpha)^{-1}$ lifts closed to $(X_{\varepsilon},\tilde{*}_2)$, where $\beta$ is any path from $*_2$ to $*_1$.  But this implies that $\gamma$ lifts closed to $(X_{\varepsilon},\tilde{*}_2)$.\end{proof}

\vspace{.1 in}
\noindent In particular, there is no ambiguity in stating that a path loop lifts closed to $X_{\varepsilon}$ without referencing the base point used to determine the $\varepsilon$-covers.  We will use this result without comment going forward.

Continuing the notation of the previous lemma, let $\pi_{cl}(X,*_1)$ and $\pi_{cl}(X,*_2)$ denote the closed lifting groups at $*_1$ and $*_2$, respectively.  For $\varepsilon > 0$, let $K_{\varepsilon}(*_1)$ and $K_{\varepsilon}(*_2)$ denote the respective kernels of $h_{\varepsilon}^1:\pi_1(X,*_1) \rightarrow \pi_{\varepsilon}(X,*_1)$ and $h_{\varepsilon}^2:\pi_1(X,*_2) \rightarrow \pi_{\varepsilon}(X,*_2)$.  Fix a path, $\alpha$, from $*_2$ to $*_1$, and define $F:\pi_1(X,*_1) \rightarrow \pi_1(X,*_2)$ by $F([\gamma]) = [\alpha \gamma \alpha^{-1}]$. 

\begin{lemma}
\label{rfg indep bp}
For a geodesic space, $X$, the following hold.
\begin{enumerate}
\item[1)] The restriction of $F$ to $K_{\varepsilon}(*_1)$ is an isomorphism onto $K_{\varepsilon}(*_2)$.
\item[2)] The restriction of $F$ to $\pi_{cl}(X,*_1)$ is an isomorphism onto $\pi_{cl}(X,*_2)$.
\item[3)] $\bar{\pi}_1(X,*_1)=\pi_1(X,*_1)/\pi_{cl}(X,*_1) \cong \pi_1(X,*_2)/\pi_{cl}(X,*_2)=\bar{\pi}_1(X,*_2)$.
\end{enumerate}
\end{lemma}

\begin{proof}
We first note that $F$ is an isomorphism; this is standard in fundamental group arguments.  Suppose $[\gamma] \in K_{\varepsilon}(*_1)$, so that $\gamma$ is a path loop at $*_1$ that lifts closed to $(X_{\varepsilon},\tilde{*}_1)$.  By Lemma \ref{first bp indep}, $\gamma$ lifts closed to $(X_{\varepsilon},\tilde{*}_2)$, implying that $\alpha \gamma \alpha^{-1}$ lifts closed to $(X_{\varepsilon},\tilde{*}_2)$.  This shows that $F$ maps $K_{\varepsilon}(*_1)$ into $K_{\varepsilon}(*_2)$.  On the other hand, if $[\lambda] \in K_{\varepsilon}(*_2)$, let $\gamma = \alpha^{-1}\lambda \alpha$, so that $F([\gamma])=[\lambda]$.  By Lemma \ref{first bp indep}, $\lambda$ lifts closed to $(X_{\varepsilon},\tilde{*}_1)$, and it follows that $\gamma$ lifts closed to $(X_{\varepsilon},\tilde{*}_1)$. Thus, $F$ maps $K_{\varepsilon}(*_1)$ onto $K_{\varepsilon}(*_2)$, proving 1.

Next, suppose $[\gamma] \in \pi_{cl}(X,*_1)$, and let $\varepsilon > 0$ be given.  Since $[\gamma] \in K_{\varepsilon}(*_1)$, $F([\gamma]) \in K_{\varepsilon}(*_2)$ by part 1.  Since $\varepsilon$ was arbitrary, this shows that $F([\gamma]) \in \pi_{cl}(X,*_2)$, and $F$ maps $\pi_{cl}(X,*_1)$ into $\pi_{cl}(X,*_2)$.  By the same reasoning, if $[\lambda] \in \pi_{cl}(X,*_2)$, then $F^{-1}([\lambda]) \in K_{\varepsilon}(*_1)$ for all $\varepsilon > 0$.  Thus, $F^{-1}([\lambda]) \in \pi_{cl}(X,*_1)$, and $F$ maps $\pi_{cl}(X,*_1)$ onto $\pi_{cl}(X,*_2)$, proving part 2.

Finally, part 3 now follows from the homomorphism result we used earlier.  If $\bar{h}_1:\pi_1(X,*_1) \rightarrow \bar{\pi}_1(X,*_1)$ and $\bar{h}_2:\pi_1(X,*_2) \rightarrow \bar{\pi}_1(X,*_2)$ denote the respective quotient maps, then the kernel of $\bar{h}_2 \circ F: \pi_1(X,*_1) \rightarrow \bar{\pi}_1(X,*_2)$ is the subgroup $F^{-1}(\pi_{cl}(X,*_2))=\pi_{cl}(X,*_1)$.  Hence, there is a unique, surjective homomorphism $f:\bar{\pi}_1(X,*_1) \rightarrow \bar{\pi}_2(X,*_2)$ satisfying $f \circ \bar{h}_1 = \bar{h}_2 \circ F$ and defined by $f([\gamma]\pi_{cl}(X,*_1)) = \bar{h}_2(F([\gamma]))$.  If $f([\gamma]\pi_{cl}(X,*_1)) = \pi_{cl}(X,*_2) \in \bar{\pi}_1(X,*_2)$, then $F([\gamma]) \in \pi_{cl}(X,*_2)$, which means that $[\gamma] \in \pi_{cl}(X,*_1)$.  That is, $f$ is injective.\end{proof}

\vspace{.1 in}
\noindent Consequently, we will treat base points very informally.  We always technically need to choose one to determine the collection of $\varepsilon$-covers $\{X_{\varepsilon}\}$, but it will play no other significant role.

Proposition \ref{characterize k0} provides some useful intuition regarding the closed lifting and revised fundamental groups.  For instance, nontrivial elements of $\bar{\pi}_1(X)$ are obviously represented by classes of path loops at $*$ that eventually lift open to $X_{\varepsilon}$ for all sufficiently small $\varepsilon > 0$.  Going further, each distinct, nontrivial element of the revised fundamental group determines a critical value in the compact case.  Suppose $[\gamma]$, $[\lambda] \in \pi_1(X)\setminus \pi_{cl}(X)$ and $[\gamma]\pi_{cl}(X) = [\lambda]\pi_{cl}(X)$, so that there is a path loop, $\sigma$, at $*$ that lifts closed to $X_{\varepsilon}$ for all $\varepsilon$ and is such that $\gamma$ is homotopic to $\lambda \sigma$.  Since $\sigma$ always lifts closed, the endpoint of the lift of $\lambda \sigma$ will just be the endpoint of the lift of $\lambda$.  Thus, $[\gamma]=[\lambda \sigma]$ implies that $\gamma$ and $\lambda$ will both lift open or closed to any given $X_{\varepsilon}$.  So, if $\gamma$ lifts open to $X_{\varepsilon}$ but closed to $X_{\delta}$ for all $\delta > \varepsilon$, thus making $\varepsilon$ a critical value, it follows that $\lambda$ has the same property.  That is, each element of $[\gamma]\pi_{cl}(X)$ determines $\varepsilon$ as a critical value.  Formally, there is a surjective map $f:\bar{\pi}_1(X) \rightarrow Cr(X) \cup \{0\}$, taking $\pi_{cl}(X)$ to $0$ and each nontrivial element to the critical value determined by that coset.  Multiple distinct cosets, of course, can determine the same critical value, so $f$ need not be injective.

For another interpretation in light of Lemma \ref{geodesic loop reduction}, we can intuitively think of $\bar{\pi}_1(X)$ as a ``large scale fundamental group.'' By this, we mean that $\bar{\pi}_1(X)$ treats as nontrivial only those fundamental group elements $[\gamma]$ that have positive diameter in the sense that $\inf_{\lambda \in [\gamma]} diam(\lambda) > 0$.  To wit, suppose $[\gamma]\pi_{cl}(X)$ is not the identity in $\bar{\pi}_1(X)$, and let $\varepsilon$ be the largest value such that $\gamma$ lifts open to $X_{\varepsilon}$.  By Corollary \ref{lollipop products}, any lollipop factorization of $[\gamma]$ must contain at least one factor, $\alpha \beta \alpha^{-1}$, where $\beta$ is a path loop that is not homotopic to any path loop lying in a ball of radius $\varepsilon$.  In other words, there is a positive lower bound of $\varepsilon$ on the diameter of the image of any representative in $[\gamma]$.

On the other hand, obtaining a general geometric picture of the elements in $\pi_{cl}(X)$ is more difficult, but Lemma \ref{geodesic loop reduction} is useful here, also.  Suppose $[\gamma] \in \pi_{cl}(X)$ is nontrivial.  Then, by Lemma \ref{geodesic loop reduction}, $[\gamma]$ can be represented by an $\varepsilon$-lollipop factorization for any $\varepsilon > 0$.  The minimal number of lollipops in each factorization may increase as $\varepsilon$ decreases, but the heads of the lollipops become arbitrarily small.  This is a generalization of a formal \textit{small path loop}, or a path loop $\gamma$ based at a point $x$ such that every neighborhood of $x$ contains a path loop that is homotopic to $\gamma$.  Note that the word `small' here is part of the name and definition, not just an adjectival descriptor.  Nontrivial small path loops have long been known as one of the obstructions to semilocal simple connectivity, though they were formally defined in the investigation of \textit{small loop spaces} by Z. Virk in \cite{V}.  By definition, $X$ is not homotopically Hausdorff if there is at least one nontrivial small path loop in $X$.

\begin{lemma}
\label{k0 contains small loops}
If $X$ is geodesic, and if $\beta$ is a small path loop based at $x \in X$, with $\alpha$ any path from $*$ to $x$, then $[\alpha \beta \alpha^{-1}] \in \pi_{cl}(X,*)$. In particular, if $\pi_{cl}(X) =1$, then $X$ is homotopically Hausdorff.
\end{lemma}

\begin{proof}
The first statement follows immediately from Corollary \ref{loop in a ball} and Lemma \ref{lollipop products}, since, by definition, there is a representative of $\beta$ based at $x$ and lying in $B_{\varepsilon}(x)$ for any $\varepsilon > 0$. If $X$ is not homotopically Hausdorff, then there is a nontrivial small path loop based at some point in $X$.  But this produces a nontrivial element of $\pi_{cl}(X)$ by the first statement.\end{proof}

\vspace{.1 in}
\noindent In particular, $\pi_{cl}(X)$ contains what Pakdaman, Torabi, and Mashayekhy call in \cite{PTM} the \textit{SG-subgroup}. This is the subgroup $\pi_1^{sg}(X) \subset \pi_1(X)$ generated by all \textit{small loop lollipops}, or path loops of the form $\alpha \beta \alpha^{-1}$, where $\alpha$ is a path beginning at $*$ and $\beta$ is a small path loop based at $\alpha(1)$.  In many cases, $\pi_1^{sg}(X)$ actually equals $\pi_{cl}(X)$ (cf. Example \ref{K example}), but this is not true in general.  The space in Example \ref{gasket example} illustrates that $\pi_{cl}(X)$ may contain elements $[\gamma]$ such that $\gamma$ is not homotopic to any finite product of small loop lollipops. We postpone the example until Section 4, since it relies on the construction of another example we give there.  It is, however, still open whether or not the converse of the second part of Lemma \ref{k0 contains small loops} holds.  Informally, is the presence of at least one nontrivial small path loop the only way in which $\pi_{cl}(X)$ can be nontrivial when $X$ is compact?  In the non-compact case, it is easy to find examples - even complete Riemannian manifolds - where $\pi_{cl}(X)$ is nontrivial while $X$ is homotopically Hausdorff.  The surface of revolution formed by rotating the graph of $f(x) = e^{-x}$ around the $x$-axis is such a space.

There is one important case, however, where these two necessarily coincide.  Cannon and Conner showed that one-dimensional spaces are automatically homotopically Hausdorff (Corollary 5.4, \cite{CC}).  Berestovskii-Plaut proved (Theorem 93, \cite{BPUU}) that when $X$ has uniform dimension at most $1$, $\Lambda$ is injective.  Since $ker \ \Lambda = \pi_{cl}(X)$, we have

\begin{lemma}
\label{k0 trivial 1d}
If $X$ is a uniformly one-dimensional geodesic space \emph{(}e.g. if $X$ is compact with topological dimension $1$\emph{)}, then $\pi_{cl}(X) = 1$ and $\bar{\pi}_1(X) = \pi_1(X)$.
\end{lemma}

There is also a metric characterization of the closed lifting group when $X$ is compact and has a universal cover, and this characterization establishes yet another connection to the work of Sormani and Wei.  In \cite{SW4}, they defined the \textit{slipping group} of a geodesic space, $X$, with a universal cover, $\hat{X}$.  This is the subgroup $\pi_{slip}(X) \subset \pi_1(X)$ generated by elements $g \in \pi_1(X)$ satisfying $\inf \{\hat{d}(x,gx) : x \in \hat{X}\} = 0$, where $\hat{d}$ is the lifted geodesic metric on $\hat{X}$ and the notation $gx$ refers to the action on $\hat{X}$ by the deck transformation corresponding to $g$.  While the definition of $\pi_{cl}(X)$ does not require a universal cover, we can show that $\pi_{cl}(X) = \pi_{slip}(X)$ when $X$ is compact geodesic and has a universal cover.

\begin{lemma}
\label{slipping lemma}
If $X$ is a compact geodesic space with a universal cover, then $\pi_{slip}(X) = \pi_{cl}(X)$.
\end{lemma}

\begin{proof}
Since a universal cover exists, $Cr(X)$ is finite.  Thus, we can choose $\varepsilon$ small enough - less than $\min Cr(X)$ - so that $X_{\varepsilon}$ is our universal cover and its deck group is $\pi_{\varepsilon}(X) \cong \pi_1(X)/K_{\varepsilon}$.  In fact, the homomorphism $h_{\varepsilon}:\pi_1(X) \rightarrow \pi_{\varepsilon}(X)$ is precisely the map that takes $[\gamma] \in \pi_1(X)$ to its corresponding deck transformation.  Moreover, Lemma \ref{univ cover first eq} will show that $\pi_{cl}(X) = K_{\varepsilon}$ in this case.

Let $[\gamma] \in \pi_1(X)$ be given, and let $[\alpha]_{\varepsilon} = h_{\varepsilon}([\gamma])$.  Suppose $[\gamma] \in \pi_{cl}(X) = K_{\varepsilon}$, so that $[\alpha]_{\varepsilon}$ is trivial.  Then $[\gamma]$ corresponds to the identity map of $X_{\varepsilon}$ and is, therefore, in $\pi_{slip}(X)$.  On the other hand, suppose $[\gamma] \in \pi_{slip}(X)$.  If $[\gamma]$ were not in $\pi_{cl}(X) = K_{\varepsilon}$, then $[\alpha]_{\varepsilon}$ would be nontrivial in $\pi_{\varepsilon}(X)$.  But the fact that $\pi_{\varepsilon}(X)$ acts discretely on $X_{\varepsilon}$ would then imply that $\hat{d}([\alpha]_{\varepsilon}[\beta]_{\varepsilon},[\beta]_{\varepsilon}) \geq \varepsilon$ for all $[\beta]_{\varepsilon} \in X_{\varepsilon}$.  This would contradict that $[\gamma] \in \pi_{slip}(X)$.\end{proof}

Next, we determine the relationship between $\bar{\pi}_1(X)$ and $\Delta(X)$.   Applying the same homomorphism theorem we used to define $\bar{h}_{\varepsilon}$, we see that there is a unique homomorphism $\bar{h}_{\Delta}:\bar{\pi}_1(X) \rightarrow \Delta(X)$ such that the following diagram commutes.
\begin{equation}
\label{comm diagram delta}
\begin{diagram}
\pi_1(X) & \rTo_{\Lambda} & \Delta(X) \\
\dTo_{\bar{h}} & \ruTo^{\bar{h}_{\Delta}} &  \\
 \bar{\pi}_1(X) &
\end{diagram}
\end{equation}
\noindent In this case, $\bar{h}_{\Delta}$ need not be surjective, because $\Lambda$ may not be. However, since $ker \ \Lambda = \pi_{cl}(X)$ instead of just containing it, $\bar{h}_{\Delta}$ is injective.  Thus, $\bar{h}_{\Delta}$ is an isomorphism if and only if it is surjective, which holds if and only if $\Lambda$ is surjective.  Recalling that $\Lambda(\pi_1(X))$ is a dense normal subgroup of $\Delta(X)$ with the inverse limit topology and that $\Delta(X) \cong \breve{\pi}_1(X)$, we have the following.

\begin{corollary}
\label{rfg embeds ufg}
If $X$ is a geodesic space, then the revised fundamental group isomorphically injects into the uniform fundamental group as a dense normal subgroup. Furthermore, $\bar{\pi}_1(X)$ isomorphically injects into $\breve{\pi}_1(X)$.
\end{corollary}

We next use the characterization of $\pi_{cl}(X)$ to establish analogs in the revised case of certain classical properties related to the fundamental group, starting with the semilocally simply connected definition.  Recall that a path connected, locally path connected space $X$ is semilocally simply connected at $x \in X$ if and only if there is some neighborhood $U$ of $x$ such that the homomorphism $i_{\ast}:\pi_1(U,x) \rightarrow \pi_1(X,x)$ induced by the inclusion $i:U \hookrightarrow X$ is trivial.  If $X$ has a simply connected universal cover, this is equivalent to saying that every path loop based at $x$ and contained in $U$ lifts closed to the universal cover and, thus, to every other cover of $X$.  With this motivation, we define the following.

\begin{definition}
\label{semi rconn}
Let $X$ be a geodesic space with base point $*$, and let $\{X_{\varepsilon}\}$ be the $\varepsilon$-covers determined by $*$.  $X$ is semilocally $r$-simply connected at $x \in X$ - or semilocally simply connected with respect to the revised fundamental group at $x$ - if there is an open ball $U$ centered at $x$ such that every path loop contained in $U$ and based at $x$ lifts closed to $X_{\varepsilon}$ for all $\varepsilon$.  $X$ is semilocally $r$-simply connected if this property holds at each $x \in X$.  If $\bar{\pi}_1(X)$ is trivial for some \emph{(}hence, every\emph{)} base point $* \in X$, we say that $X$ is $r$-simply connected.
\end{definition}

\noindent  The algebraic formulation of semilocal simple connectivity carries over to this definition, as well.

\begin{lemma}
\label{equiv rconn}
Let $X$ be a geodesic space.  For any $x \in X$, let $\pi_{cl}(X,x)$ denote the closed lifting group in $\pi_1(X,x)$, and let $i:U \hookrightarrow X$ denote the inclusion map of a subset $U \subset X$.  The following are equivalent.
\begin{enumerate}
\item[1)] $X$ is semilocally $r$-simply connected at $x$.
\item[2)] There is an open ball $U$ centered at $x$ such that the homomorphism $i_{\ast}:\pi_1(U,x) \rightarrow \pi_1(X,x)$ has image in $\pi_{cl}(X,x)$.
\item[3)] There is an open ball $U$ centered at $x$ such that $\bar{h} \circ i_{\ast}:\pi_1(U,x) \rightarrow \bar{\pi}_1(X,x)$ is trivial.
\end{enumerate}
\noindent Moreover, if $\{X_{\varepsilon}\}_{\varepsilon > 0}$ are the $\varepsilon$-covers determined by any fixed base point, $* \in X$, then $X$ is $r$-simply connected if and only if every path loop in $X$ lifts closed to $X_{\varepsilon}$ for all $\varepsilon$.
\end{lemma}

\begin{proof}
That $1$, $2$, and $3$ are equivalent follows clearly from the definitions and Lemma \ref{first bp indep}.  If every path loop in $X$ lifts closed to $X_{\varepsilon}$ for all $\varepsilon$, then we immediately have that $\pi_1(X,*) = \pi_{cl}(X,*)$ and $\bar{\pi}_1(X,*)$ is trivial.  Conversely, if $\bar{\pi}_1(X,*)$ is trivial, then every path loop at $*$ lifts closed to $X_{\varepsilon}$ for all $\varepsilon > 0$.  If $\gamma$ is a path loop based at $x \neq *$, then the path loop $\alpha \gamma \alpha^{-1}$ - where $\alpha$ is a fixed path from $*$ to $x$ - will lift closed to $X_{\varepsilon}$ for all $\varepsilon$.  Thus, $\gamma$ will also.\end{proof}

\vspace{.1 in}
\noindent In the next section, we show that a universal cover of a compact geodesic space $X$ is $r$-simply connected and that for such a cover to exist it is necesary and sufficient that $X$ be semilocally $r$-simply connected.

Next, let $X$ and $Y$ be geodesic spaces, and let $f:(X,x) \rightarrow (Y,y)$ be a pointed, \emph{uniformly} continuous function.  For this particular group of results, we avoid using $*$ as the base point, since it will have another usage.  Given that the Berestovskii-Plaut construction exploits the uniform structure of a metric space, it is necessary to work with uniformly continuous functions.  If $X$ is compact, then this condition, of course, follows without assumption.  There is an induced homomorphism $f_{\ast}:\pi_1(X,x) \rightarrow \pi_1(Y,y)$ defined by $[\gamma] \mapsto [f \circ \gamma]$.  Let $\bar{h}_X:\pi_1(X,x) \rightarrow \bar{\pi}_1(X,x)$ and $\bar{h}_Y:\pi_1(Y,y) \rightarrow \bar{\pi}_1(Y,y)$ denote the corresponding quotient maps.

Now, suppose $[\gamma] \in \pi_{cl}(X,x)$, and let $\varepsilon > 0$ be given.  Since $f$ is uniformly continuous, there is some $\delta > 0$ such that $d(x_1,x_2) < \delta \Rightarrow d(f(x_1),f(x_2)) < \varepsilon$.  Moreover, $\gamma$ lifts closed to $X_{\delta/3}$ by assumption.  Thus, there is a representative of $[\gamma]$, say $\gamma'$, that is a product of $\frac{\delta}{2}$-lollipops, $\alpha_1 \beta_1 \alpha_1^{-1} \cdots \alpha_k \beta_k \alpha_k^{-1}$.  Let $\lambda = f \circ \gamma' = \omega_1 \sigma_1 \omega_1^{-1} \cdots \omega_k \sigma_k \omega_k^{-1}$, where $\omega_i = f \circ \alpha_i$ and $\sigma_i = f \circ \beta_i$ for $i=1,\dots,k$.  Then $\lambda$ is a lollipop factorization of $[\lambda] = [f \circ \gamma] \in \pi_1(Y,y)$.  Fix any $i=1,\dots,k$, and let $\sigma_i(t)$ be any point on $\sigma_i$.  Since $\beta_i$ lies in a ball of radius $\frac{\delta}{2}$, we have
\[d(\beta_i(0),\beta_i(t)) < \delta \ \Rightarrow \ d(\sigma_i(0),\sigma_i(t)) = d\Bigl(f\bigl(\beta_i(0)\bigr),f\bigl(\beta_i(t)\bigr)\Bigr) < \varepsilon. \]
\noindent So, each $\sigma_i$ lies in $B_{\varepsilon}(\sigma_i(0))$, and $\omega_1\sigma_1\omega_1^{-1}\cdots\omega_k \sigma_k \omega_k^{-1}$ is an $\varepsilon$-lollipop factorization of $[f\circ \gamma]$.  By Corollary \ref{lollipop products}, $\lambda$ lifts closed to $Y_{\varepsilon}$.  Finally, since $\varepsilon$ was arbitrary, we can conclude that $[\lambda] = [f \circ \gamma] \in \pi_{cl}(Y,y)$.  In other words, we have shown that $f_{\ast}$ takes $\pi_{cl}(X,x)$ into $\pi_{cl}(Y,y)$, or that $\pi_{cl}(X,x) \subseteq ker \ (\bar{h}_Y \circ f_{\ast})$.  Applying our usual homomorphism result, we have shown the following.

\begin{lemma}
\label{induced hom}
Let $f:(X,x) \rightarrow (Y,y)$ be a uniformly continuous function between geodesic spaces \emph{(}e.g. $X$ compact and $f$ continuous\emph{)}.  Then there is a unique induced homomorphism $f_{\sharp}:\bar{\pi}_1(X,x) \rightarrow \bar{\pi}_1(Y,y)$ that is defined by 
\[f_{\sharp}([\gamma]\pi_{cl}(X,x))= \bigl( f_{\sharp} \circ \bar{h}_X\bigr)([\gamma]) \defeq \bigl(\bar{h}_Y \circ f_{\ast}\bigr)([\gamma]) =[f \circ \gamma]\pi_{cl}(Y,y)\]
\noindent and is such that the following diagram commutes.
\begin{equation}
\label{induced}
\begin{diagram}
\pi_{1}(X,x) & \rTo_{f_{\ast}} & \pi_{1}(Y,y) \\
\dTo^{\bar{h}_X} & & \dTo_{\bar{h}_Y} &  \\
\bar{\pi}_1(X,x) & \rTo^{f_{\sharp}} & \bar{\pi}_1(Y,y)
\end{diagram}
\end{equation}
\end{lemma}

\vspace{.1 in}
\begin{corollary}
\label{homotopy equiv}
Assume the hypotheses of the preceding lemma.  If $f$ is also a homotopy equivalence with a uniformly continuous homotopy inverse $g:Y \rightarrow X$, then $f_{\sharp}$ is an isomorphism.  That is, uniformly homotopy equivalent geodesic spaces - in particular, homotopy equivalent, compact geodesic spaces - have isomorphic revised fundamental groups.
\end{corollary}

\begin{proof}
If $f$ is a homotopy equivalence, then $f_{\ast}$ is an isomorphism, and the above diagram immediately shows that $f_{\sharp}$ is surjective.  If $g(y)=x'$ and $g_{\ast}:\pi_1(Y,y) \rightarrow \pi_1(X,x')$ is the induced isomorphism, it follows that there is a path $\alpha$ from $x'$ to $x$ such that $(g_{\ast} \circ f_{\ast})([\gamma]) = [\alpha \gamma \alpha^{-1}]$ for any $[\gamma] \in \pi_1(X,x)$.  This is just a consequence of the classical proof that homotopy equivalence induces an isomorphism of fundamental groups.  Next, let $\bar{h}_X:\pi_1(X,x) \rightarrow \bar{\pi}_1(X,x)$ and $\bar{h}'_X:\pi_1(X,x') \rightarrow \bar{\pi}_1(X,x')$ denote the quotient maps at the respective base points in $X$, and let $g_{\sharp}:\bar{\pi}_1(Y,y) \rightarrow \bar{\pi}_1(X,x')$ be the homomorphism induced by $g$. Note that $\bar{h}'_X \circ g_{\ast} \circ f_{\ast} = g_{\sharp} \circ f_{\sharp} \circ \bar{h}_X$, since
\[(g_{\sharp} \circ f_{\sharp} \circ \bar{h}_X)([\gamma])=g_{\sharp}([f\circ \gamma]\pi_{cl}(Y,y)) =[g\circ f \circ \gamma]\pi_{cl}(X,x')  \]
\[(\bar{h}'_X \circ g_{\ast} \circ f_{\ast})([\gamma]) = \bar{h}'_X([g \circ f \circ \gamma])=[g \circ f \circ \gamma]\pi_{cl}(X,x').\]
\noindent  Finally, suppose $[\gamma] \in \pi_1(X,x)$ and $f_{\sharp}([\gamma]\pi_{cl}(X,x)) = (f_{\sharp} \circ \bar{h}_X)([\gamma]) = \pi_{cl}(Y,y) \in \bar{\pi}_1(Y,y)$.  Then $(g_{\sharp} \circ f_{\sharp} \circ \bar{h}_X)([\gamma]) = \pi_{cl}(X,x') \in \bar{\pi}_1(X,x')$, implying that $(\bar{h}'_X \circ g_{\ast} \circ f_{\ast})([\gamma]) = \pi_{cl}(X,x')$.  But this implies that $g_{\ast}(f_{\ast}([\gamma])) \in \pi_{cl}(X,x')\subset \pi_1(X,x')$, from which we have $[\alpha \gamma \alpha^{-1}] \in \pi_{cl}(X,x')$.  By Lemma \ref{rfg indep bp}, it follows that $[\gamma] \in \pi_{cl}(X,x)$, or $[\gamma]\pi_{cl}(X,x) = \pi_{cl}(X,x)$. Hence, $f_{\sharp}$ is injective.\end{proof}

\vspace{.1 in}
\section{Universal Covers and Group Geometry}

Since $\{K_{\varepsilon}\}_{\varepsilon > 0}$ is a nested decreasing set, there is a well-defined notion of $K_{\varepsilon}$ being \textit{eventually constant}. This just means there is a $\delta$ so that $K_{\varepsilon} = K_{\varepsilon'}$ for all $\varepsilon$, $\varepsilon' \leq \delta$.  Equivalently, there is some $\delta > 0$ such that $\pi_{cl}(X) = \bigcap_{\varepsilon > 0} K_{\varepsilon} = K_{\delta}$.  This turns out to be an equivalent way of saying that the critical spectrum is finite.

\begin{lemma}
\label{univ cover first eq}
If $X$ is a compact geodesic space, then the following are equivalent.
\begin{enumerate}
\item[1)] $Cr(X)$ is finite.
\item[2)] $\{K_{\varepsilon}\}_{\varepsilon > 0}$ is eventually constant. Specifically, $\pi_{cl}(X) = K_{\varepsilon}$ for $0 < \varepsilon \leq \min Cr(X)$.
\item[3)] $\bar{h}_{\varepsilon}:\bar{\pi}_1(X) \rightarrow \pi_{\varepsilon}(X)$ is an isomorphism for all $0<\varepsilon \leq \min Cr(X)$.
\item[4)] $\bar{h}_{\varepsilon}:\bar{\pi}_1(X) \rightarrow \pi_{\varepsilon}(X)$ is an isomorphism for some $\varepsilon$.
\end{enumerate}
\end{lemma}

\begin{proof}
$(1 \Rightarrow 2)$ Let $\varepsilon'= \min Cr(X)$. Lemma 21 in \cite{PWetal} shows that if $\Phi_{\varepsilon \delta}$ is non-injective for $\delta < \varepsilon$, then there is a critical value in $[\delta,\varepsilon)$. So, if $\delta < \varepsilon \leq \varepsilon'$, and if we had $K_{\delta} \subsetneq K_{\varepsilon}$, it would follow from commutative diagram (\ref{comm diagram 1}) that $\Phi_{\varepsilon \delta}$ is non-injective, giving a critical value below $\varepsilon'$.  This is a contradiction.

$(2 \Rightarrow 3)$ Fix $0 < \varepsilon \leq \min Cr(X)$, and suppose $\bar{h}_{\varepsilon}([\gamma]\pi_{cl}(X)) = \bar{h}_{\varepsilon}([\gamma]K_{\varepsilon}) = \tilde{*} \in \pi_{\varepsilon}(X)$.  Then $h_{\varepsilon}([\gamma]) = \tilde{*}$, and $[\gamma]\in K_{\varepsilon}=\pi_{cl}(X)$. So, $ker \ \bar{h}_{\varepsilon}$ is trivial.  The fact that $3$ implies $4$ is obvious.

$(4 \Rightarrow 1)$ Suppose $\bar{h}_{\varepsilon}$ is an isomorphism for some $\varepsilon$, and let $0 < \delta <\varepsilon$ be given. As with diagram (\ref{comm diagram 1}), the following also commutes.
\begin{diagram}
\pi_{\delta}(X) & \rTo_{\Phi_{\varepsilon \delta}} & \pi_{\varepsilon}(X) \\
\uTo^{\bar{h}_{\delta}} & & \uTo_{\bar{h}_{\varepsilon}} &  \\
\bar{\pi}_1(X) & \rTo^{\iota} & \bar{\pi}_1(X)
\end{diagram}
\noindent Since $\bar{h}_{\varepsilon} \circ \iota$ is an isomorphism and $\bar{h}_{\delta}$ and $\Phi_{\varepsilon \delta}$ are surjective, this forces each of the latter maps to be injective, also.  Finally, if $0 < \tau < \delta \leq \varepsilon$, then $\Phi_{\varepsilon \tau} =\Phi_{\varepsilon \delta} \circ \Phi_{\delta \tau} \Rightarrow \Phi_{\delta \tau} = \Phi_{\varepsilon \delta}^{-1} \circ \Phi_{\varepsilon \tau}$, and the right hand side is an isomorphism.  Thus, $\Phi_{\delta \tau}$ is an isomorphism for all $0 < \tau < \delta \leq \varepsilon$, showing that there are no critical values below $\varepsilon$.\end{proof}

\vspace{.1 in}
To prove Theorem \ref{main theorem 1}, we need a technical result giving sufficient conditions for $X$ to be semilocally $r$-simply connected.  The lemma relies on a result of Cannon and Conner regarding an Artinian property of fundamental groups.  For reference purposes, we state the portion of that theorem that is needed here as a lemma, specializing it to metric spaces.

\begin{lemma}[\textbf{Theorem 4.4, \cite{CC}}]
\label{cc lemma}
Let $X$ be a metric space, $x$ a base point, and $f:\pi_1(X,x) \rightarrow L$ a homomorphism to a group, $L$.  For $n=1,2,\dots$, let $B_n$ be the open ball $B_{1/n}(x)$, and let $G_n$ be the image of the natural map $\pi_1(B_n,x) \rightarrow \pi_1(X,x)$.  If $L$ is countable, then the sequence $f(G_1) \supseteq f(G_2) \supseteq \cdots$ is eventually constant.
\end{lemma}

\begin{proposition}
\label{local lift property}
Let $X$ be a geodesic space, and suppose $\bar{\pi}_1(X)$ is countable.  Then for each $x \in X$ there is an open ball centered at $x$, $B_{r_x}(x)$, with the following property: if $\gamma$ is a path loop in $B_{r_x}(x)$, not necessarily based at $x$, then it lifts closed to $X_{\varepsilon}$ for all $\varepsilon > 0$.  In particular, if $\bar{\pi}_1(X)$ is countable, then $X$ is semilocally $r$-simply connected.
\end{proposition}

\begin{proof}
In keeping with our notation, we will denote the base point by $*$ instead of $x$.  The countability of $\bar{\pi}_1(X)$ is independent of the base point by Lemma \ref{rfg indep bp}.  We apply Lemma \ref{cc lemma} with $L=\bar{\pi}_1(X,*)$, $f = \bar{h}:\pi_1(X,*) \rightarrow \bar{\pi}_1(X,*)$, and $G_n$ the image of $\pi_1(B_n,*) \rightarrow \pi_1(X,*)$, from which it follows that $\bar{h}(G_1) \supseteq \bar{h}(G_2) \supseteq \cdots$ is eventually constant.

We first show that $\bar{h}(G_n) = \{g\pi_{cl}(X) \: : \: g \ \text{has a representative in }B_n \ \text{based at }*\}$ for each $n$.  If $g \in \pi_1(X,*)$ and has a representative in $B_n$ based at $*$, then $g$ is in the image of the map $\pi_1(B_n,*) \rightarrow \pi_1(X,*)$.  That is, $g \in G_n$, implying that $g\pi_{cl}(X)=\bar{h}(g) \in \bar{h}(G_n)$.  Conversely, let $g\pi_{cl}(X) \in \bar{h}(G_n)$ be given, so that $g\pi_{cl}(X) = \bar{h}(g')$ for some $g' \in G_n$.  Then $g\pi_{cl}(X) = g'\pi_{cl}(X)$, and $g'$ has a representative in $B_n$ based at $*$, proving the equality.

Now, since the nested sequence $\{\bar{h}(G_n)\}$ is eventually constant, we can choose $m \in \mathbb N$ large enough so that $\bar{h}(G_m) = \bigcap_n \bar{h}(G_n)$.  We claim that $B_m = B_{1/m}(*)$ is the desired ball.

First, suppose $\gamma$ is a path loop based at $*$ and lying in $B_m$, and let $g = [\gamma] \in \pi_1(X,*)$.  Then $\bar{h}(g) = g\pi_{cl}(X) \in \bar{h}(G_m)$.  Let $n$ be fixed but arbitrary, so that $\bar{h}(g) \in \bar{h}(G_n)$. This implies that $g\pi_{cl}(X) = g'\pi_{cl}(X)$, where $g'$ has a representative in $B_n$ based at $*$, and this further implies that $g=g'l$ for some $l \in \pi_{cl}(X)$.  This means that $\gamma$ is homotopic to a product $\alpha \beta$, where $\alpha$ is a representative of $g'$ in $B_n$ based at $*$ and $\beta \in l$.  But $\beta$ lifts closed to $X_{\varepsilon}$ for all $\varepsilon > 0$.  Moreover, since $\alpha$ is a path loop at $*$ lying in $B_{1/n}(*)$, $\alpha$ will lift closed to $X_{1/n}$ by Lemma \ref{loop in a ball}.  Hence, $\gamma$ lifts closed to $X_{1/n}$, and this holds for all $n$.  Since the $\varepsilon$-covers are monotone (i.e. $\delta < \varepsilon$ implies $X_{\delta}$ covers $X_{\varepsilon}$), we see that $\gamma$ lifts closed to $X_{\varepsilon}$ for all $\varepsilon > 0$.

This proves the result for path loops based at $*$ in $B_m$.  To finish the proof, let $\gamma$ be any path loop in $B_m$, not necessarily based at $*$.  Let $\lambda$ be a minimal geodesic from $*$ to the initial/terminal point of $\gamma$, and consider the $\frac{1}{m}$-lollipop $\lambda \gamma \lambda^{-1}$.  The geodesic $\lambda$ must lie in $B_m(*)$, so the whole path loop $\lambda \gamma \lambda^{-1}$ does, as well.  By the previous conclusion, $\lambda \gamma \lambda^{-1}$ lifts closed to $X_{\varepsilon}$ for each $\varepsilon > 0$, which implies that $\gamma$ does, also.  Finally, $*$ was arbitrary, so, letting $*$ range over all of $X$, we obtain the desired result.\end{proof}

\begin{corollary}
\label{local lift property2}
If $X$ is a compact geodesic space with a countable revised fundamental group, $\bar{\pi}_1(X)$, then there is a positive real number $r$ such that the following property holds:  if $\gamma$ is any path loop in $X$ that is contained in an open ball of radius $r$, then $\gamma$ lifts closed to $X_{\varepsilon}$ for all $\varepsilon > 0$.  In particular, if $X$ is compact and $\bar{\pi}_1(X)$ is countable, then $X$ is uniformly semilocally $r$-simply connected.
\end{corollary}

\begin{proof}
For each $x \in X$, let $r_x > 0$ be as in the previous proposition, so that any path loop lying in $B_{r_x}(x)$ lifts closed to $X_{\varepsilon}$ for all $\varepsilon > 0$.  Choose a positive number $r$ so that $2r$ is less than a Lebesgue number for the covering $\{B_{r_x}(x)\}$.  Then any ball of radius $r$ lies inside some element of this covering.  Thus, if $\gamma$ is any path loop in $X$ that lies in a ball of radius $r$, it will lift closed to $X_{\varepsilon}$ for all $\varepsilon > 0$.\end{proof}

\vspace{.1 in} Building on the theme of characterizing local properties of $X$ in terms of lifts of path loops, we can use this idea to classify the topological singularities in a compact geodesic space.

\begin{definition}
\label{regular singular point}
A point $x$ in a geodesic space $X$ is regular if $X$ is semilocally simply connected at $x$, in the traditional sense.  A point that is not regular is called singular.  A singular point $x$ is degenerate if every neighborhood of $x$ contains a nontrivial path loop based at $x$ that lifts closed to $X_{\varepsilon}$ for all $\varepsilon > 0$.  A singular point will be called sequentially singular if there is a sequence of path loops based at $x$, say $\gamma_n$, and a strictly decreasing sequence of positive numbers, $r_n \searrow 0$, such that for each $n$ $\gamma_n$ lies in $B_{r_n}(x)$ \emph{(}hence, lifts closed to $X_{r_n}$\emph{)} but lifts open to $X_{r_{n+1}}$.
\end{definition}

\noindent This definition does, indeed, capture all singularities of $X$.

\begin{lemma}
If $X$ is geodesic and $x \in X$ is singular, then it is either degenerate or sequentially singular.
\end{lemma}

\begin{proof}
If $x$ is degenerate, then we are done.  So, we assume that $x$ is not degenerate, and we will show that it is sequentially singular.  Since $x$ is not degenerate, there is some $r_1 > 0$ such that every nontrivial path loop based at $x$ and lying in $B_{r_1}(x)$ - and there is at least one, since $x$ is singular - eventually lifts open to $X_{\varepsilon}$ for some $\varepsilon$.  We may assume without loss of generality that $r_1 < 1$.  Choose any such path loop, label it $\gamma_1$, and let $\varepsilon_1$ be the largest value such that $\gamma_1$ lifts open to $X_{\varepsilon_1}$.  Note that we must have $\varepsilon_1 < r_1$, since the fact that $\gamma_1$ lies in $B_{r_1}(x)$ implies that it lifts closed to $X_{r_1}$.

Next, choose $r_2$ such that $0 < r_2 < \min \bigl\{\frac{1}{2},\varepsilon_1\bigr\} < r_1$.  There is a nontrivial path loop, $\gamma_2$, based at $x$ and lying in $B_{r_2}(x)$, and we let $\varepsilon_2$ be the largest value such that $\gamma_2$ lifts open to $X_{\varepsilon_2}$.  As before, we must have $\varepsilon_2 < r_2$, and note that $\gamma_1$ lifts open to $X_{r_2}$.  We continue  this process, inductively, choosing $0 < r_n < \min \bigl\{\frac{1}{n},\varepsilon_{n-1}\bigr\} < r_{n-1}$ at each step.  Then $\gamma_n$ will lie in $B_{r_n}(x)$ but lift open to $X_{r_{n+1}}$, and $r_n \searrow 0$ since $r_n < \frac{1}{n}$.  Thus, $x$ is sequentially singular.\end{proof}

Of course, a point may be both degenerate and sequentially singular.  Note that if $x$ is a point at which $X$ is not homotopically Hausdorff, so that there is a nontrivial small path loop based at $x$, then $x$ is degenerate by Corollary \ref{loop in a ball}. The reason for the name `sequentially singular' is that such points produce a sequence of critical values converging to $0$.  They can be thought of as topological singularities of the type in the Hawaiian earring, but only with regard to the lift properties of the path loops making up the sequence $\{\gamma_n\}$. It is not necessarily true, for instance, that the presence of a sequentially singular point implies that the Hawaiian earring fundamental group isomorphically embeds into $\pi_1(X)$; see Example \ref{infinite torus}.  The correct comparison is made more precise in the following lemma.

\begin{lemma}
\label{sequentially singular}
If $X$ is geodesic and $x\in X$ is a sequentially singular point, then $Cr(X)$ is infinite and $\inf Cr(X) = 0$.  Conversely, if $X$ is compact geodesic and $Cr(X)$ is infinite, then there is a sequentially singular point in $X$.
\end{lemma}

\begin{proof}
The first statement simply follows from the definition of sequentially singular.  There is a critical value in the interval $[r_2,r_1)$, detecting where $\gamma_1$ is unraveled.  Likewise, there is a critical value in $[r_3,r_2)$, detecting where $\gamma_2$ is unraveled.  Continuing, since $r_n \searrow 0$, we obtain a strictly decreasing sequence of critical values converging to $0$.

For the second statement, assume $X$ is compact.  Since $Cr(X)$ is closed, discrete, and bounded above in $\mathbb R_+$, the fact that it is infinite means that it must be a sequence of critical values strictly decreasing to $0$.  Let $\varepsilon_n$ be this sequence.  Using the fact that critical values detect where path loops are unraveled, by successively applying Lemma \ref{geodesic loop reduction}, we can inductively construct a sequence of path loops, $\beta_n$, based at points, $x_n$, such that each $\beta_n$ lies in a ball of radius $2\varepsilon_n$ but lifts open to $X_{\varepsilon_n}$.  Since $X$ is compact, we can choose a convergent subsequence of $\{x_n\}$ so that $d(x_n,x)$ is strictly decreasing and satisfies $d(x_n,x) < \frac{1}{n}$.  By reindexing if necessary, we will simply assume that $x_n \rightarrow x$ and has this property.  For each $n$, let $\alpha_n$ be a minimal geodesic from $x$ to $x_n$, and let $\gamma_n = \alpha_n \beta_n \alpha_n^{-1}$.  Then $\gamma_n$ lies in $B_{4\varepsilon_n + 1/n}(x)$ but lifts open to $X_{\varepsilon_n}$.  Choose a subsequence $\{n_k\} \subset \mathbb N$ so that 
\[4\varepsilon_{n_1} + \frac{1}{n_1} > \varepsilon_{n_1} > 4\varepsilon_{n_2} + \frac{1}{n_2} > \varepsilon_{n_2} > 4\varepsilon_{n_3} + \frac{1}{n_3} > \varepsilon_{n_3} > \cdots,\]
\noindent and let $r_k = 4\varepsilon_{n_k} + \frac{1}{n_k}$. Then $\gamma_{n_k}$ lies in $B_{r_k}(x)$ but lifts open to $X_{r_{k+1}}$.\end{proof}

\vspace{.1 in} 
\noindent In light of the fact that a compact geodesic space has a finite critical spectrum if and only if it has a universal cover, the previous result tells us that if $X$ has a universal cover, then its singular points, if any, are all degenerate and non-sequentially singular.  Thus, the degenerate points, in this case, are the only obstructions to $X$ being semilocally simply connected.

\begin{lemma}
\label{cover properties}
If a compact geodesic space, $X$, has a universal cover, $\hat{X}$, then the following hold: \emph{1)} $\hat{X}$ is $r$-simply connected; \emph{2)} its covering group is $\pi_{cl}(X)$; \emph{3)} its deck group is $\bar{\pi}_1(X)$; \emph{4)} and its singular points, if any, are degenerate and non-sequentially singular.
\end{lemma}

\begin{proof}
We fix a base point $*$.  Since $X$ has a universal cover, $Cr(X)$ is finite, and we can take $\hat{X}$ to be $X_{\varepsilon}$ for any $0 < \varepsilon \leq \min Cr(X)$.  The covering group of $X_{\varepsilon}$ is $K_{\varepsilon}$, which equals $\pi_{cl}(X)$ by Lemma \ref{univ cover first eq}.  The deck group of $X_{\varepsilon}$ is $\pi_{\varepsilon}(X)$, and recall that the groups $\pi_{\delta}(X)$ are all isomorphic for $0 < \delta \leq \min Cr(X)$.  Since $\bar{\pi}_1(X) = \pi_1(X)/\pi_{cl}(X) = \pi_1(X)/K_{\varepsilon} \cong \pi_{\varepsilon}(X)$, $\bar{\pi}_1(X)$ is the deck group of the universal cover.

For parts 1 and 4, we take, as usual, $\tilde{*} \in X_{\varepsilon}$ as our base point in the universal cover.  To prove that $X_{\varepsilon}$ is $r$-simply connected, we only need to show that $\pi_1(X_{\varepsilon},\tilde{*}) \subset \pi_{cl}(X_{\varepsilon},\tilde{*})$, where $\pi_{cl}(X_{\varepsilon})$ is the closed lifting group of the universal cover at $\tilde{*}$.  Recall that none of the fundamental constructions in Section 3 required compactness, so they apply equally well to the universal cover.  Suppose there is a path loop $\tilde{\gamma}$ at $\tilde{*} \in X_{\varepsilon}$ that lifts closed to $(X_{\varepsilon})_{\delta}$ but open to $(X_{\varepsilon})_{\tau}$ for some $0 < \tau < \delta$.  Let $\gamma = \varphi_{\varepsilon}(\tilde{\gamma})$, so that $[\gamma] \in \pi_{cl}(X)$.  Choose $\eta < \min \{\tau, \varepsilon/2\}$.  Since $\gamma$ lifts closed to $X_{2\eta/3}$, it is homotopic to a product of $\eta$-lollipops, $\alpha_1 \beta_1 \alpha_1^{-1} \cdots \alpha_k \beta_k \alpha_k^{-1}$, with each $\beta_i$ lying in a ball of radius $\eta$.  Let $\tilde{\gamma}'$ denote the lift of this product to $\tilde{*}$, so that $\tilde{\gamma}$ and $\tilde{\gamma}'$ are homotopic.  The cover $\varphi_{\varepsilon}:X_{\varepsilon} \rightarrow X$ is an isometry on $\frac{\varepsilon}{2}$-balls (Theorem 2.2.5, \cite{W2}), so $\tilde{\gamma}'$ is not only closed but the lift of each $\beta_i$ lies in a ball of radius $\eta < \tau$.  In other words, $\tilde{\gamma}$ is homotopic to a product of $\tau$-lollipops, which, by Corollary \ref{lollipop products}, contradicts that $\tilde{\gamma}$ lifts open to $(X_{\varepsilon})_{\tau}$.  Hence, every path loop at $\tilde{*}$ lifts closed to $(X_{\varepsilon})_{\delta}$ for all $\delta >0$, which means that $\pi_1(X_{\varepsilon},\tilde{*}) = \pi_{cl}(X_{\varepsilon}, \tilde{*})$ and that $X_{\varepsilon}$ has no sequentially singular points.\end{proof}

\vspace{.1 in}
We can now prove our primary theorem.

\begin{proof}[\textbf{Proof of Theorem \emph{\ref{main theorem 1}}}]
We have already remarked in the introduction that $1 \Leftrightarrow 2$. That $2 \Rightarrow 3iii$ follows from Lemma \ref{univ cover first eq} and the fact (see Section 2) that the $\varepsilon$-groups are always finitely presented.  The equivalence of $2$ and $5$ follows from Lemma \ref{sequentially singular}.  

To see that $3i$ implies $2$, assume $\bar{\pi}_1(X)$ is countable, and suppose, toward a contradiction, that $Cr(X)$ is not finite.  Let $r>0$ be as in the conclusion of Corollary \ref{local lift property2}, so that any path loop lying in a ball of radius $r$ lifts closed to every $X_{\varepsilon}$.  If $Cr(X)$ is not finite, then there is a critical value, $\varepsilon$, such that $\varepsilon < \frac{2r}{3}$.  This means that we can find a path loop at $*$ that lifts open to $X_{\varepsilon}$ but closed to $X_{\delta}$ for all $\delta > \varepsilon$.  Choose $\delta$ so that $\varepsilon < \delta < \frac{2r}{3}$.  Then Lemma \ref{geodesic loop reduction} shows that $\gamma$ is path homotopic to a product of $\frac{3\delta}{2}$-lollipops based at $*$, say $\alpha_1 \beta_1 \alpha_1^{-1} \alpha_2 \beta_2 \alpha_2^{-1} \cdots \alpha_n \beta_n \alpha_n^{-1}$. Since $\frac{3\delta}{2} < r$, each such $\beta_i$ lifts closed to $X_{\tau}$ for all $\tau > 0$, which means that $\gamma$ lifts closed to $X_{\varepsilon}$, a contradiction.  Hence, $Cr(X)$ is finite.  This shows the equivalence of $1$, $2$, $3i$, $3ii$, $3iii$, and $5$.

The implication $3i \Rightarrow 6$ follows from Corollary \ref{local lift property2}.  On the other hand, if $X$ is semilocally $r$-simply connected, then we can use compactness and the fact that geodesic balls are path connected to find a uniform positive radius, $r$, so that any path loop lying in a ball of radius $r$ lifts closed to $X_{\varepsilon}$ for all $\varepsilon$.  Then the argument in the previous paragraph goes through without change to show that $X$ has no critical values less than $\frac{2r}{3}$.  Thus, $6$ implies $2$.

Next, note that $4i$ implies $3i$, since $\bar{h}_{\Delta}:\bar{\pi}_1(X) \rightarrow \Delta(X)$ is injective.  Hence, $4i$ implies $2$. Conversely, suppose $2$ holds, so that, in particular, $\bar{\pi}_1(X)$ is finitely presented.  The finiteness of $Cr(X)$ implies that the $\varepsilon$-covers stabilize and that the UU-cover is homeomorphic to $X_{\varepsilon}$ for sufficiently small $\varepsilon$.  Hence, the UU-cover is path connected.  This means that $\Lambda:\pi_1(X) \rightarrow \Delta(X)$ is surjective, and, thus, $\bar{h}_{\Delta}$ is an isomorphism.  So $\Delta(X)$ is finitely presented, and $2 \Rightarrow 4iii \Rightarrow 4ii \Rightarrow 4i$.  

Finally, we showed $\bar{\pi}_1(X) \cong \Delta(X)$ in preceding argument.  The previous lemma establishes the rest of the conclusions in the last statement of Theorem \ref{main theorem 1}.\end{proof}

\begin{corollary}
\label{full clgroup}
If $X$ is a compact geodesic space, then the following are equivalent.
\begin{enumerate}
\item[1)] $X$ is its own universal cover.
\item[2)] $\pi_{cl}(X) = \pi_1(X)$.
\item[3)] $Cr(X) = \emptyset$.
\end{enumerate}
\end{corollary}

\begin{proof}
Suppose $X$ is its own universal cover, and let $\gamma$ be any path loop at $*$.  Then, since $\gamma$ obviously lifts closed to the universal cover, it lifts closed to $X_{\varepsilon}$ for every $\varepsilon > 0$ by the universal covering property.  Hence, $[\gamma] \in \pi_{cl}(X)$, and this shows that $1 \Rightarrow 2$.  If $2$ holds, then every path loop at $*$ lifts closed to every $X_{\varepsilon}$.  A critical value would indicate that some path loop, however, eventually lifts open to some $X_{\varepsilon}$.  Thus, $2$ implies $3$.  Finally, suppose
$3$ holds.  Then, not only does $X$ have a universal cover by Theorem \ref{main theorem 1}, but the $\varepsilon$-covers never change as $\varepsilon$ decreases and are, thus, all equivalent.  Since $X_{\varepsilon} = X$ for $\varepsilon > diam(X)$, it follows that the universal cover, which is just $X_{\varepsilon}$ for small enough $\varepsilon$, is simply $X$.\end{proof}

\begin{proposition}
\label{charac simply conn covers}
Let $X$ be a compact geodesic space.  Then $X$ has a simply connected cover if and only if $Cr(X)$ is finite and $\pi_{cl}(X) = 1$.  Moreover, for these conditions to hold it is necessary that $\pi_1(X)$, $\bar{\pi}_1(X)$, and $\Delta(X)$ be mutually isomorphic.
\end{proposition}

\begin{proof}
In our present language, Plaut and the author showed (Theorem 27, \cite{PW1}) that if $X$ has a simply connected cover then $h_{\varepsilon}:\pi_1(X) \rightarrow \pi_{\varepsilon}(X)$ is an isomorphism for some - equivalently, all - sufficiently small $\varepsilon$.  So, if $X$ has a simply connected cover, then $Cr(X)$ is finite by Theorem \ref{main theorem 1}, and $\{K_{\varepsilon}\}_{\varepsilon > 0}$ is not only eventually constant but eventually trivial.  If $Cr(X)$ is finite and $\pi_{cl}(X) = 1$, then the covering groups $K_{\varepsilon}$ are eventually trivial, meaning that $X_{\varepsilon}$ is simply connected for all sufficiently small $\varepsilon$.

If the conditions in the first statement hold, then we have already noted that the UU-cover is path connected.  Since $ker \ \Lambda = \pi_{cl}(X)$, it follows that $\Lambda$ is an isomorphism, implying that $\bar{h}_{\Delta}$ is, as well.  The fact that $\bar{h}$ is injective follows because $\pi_{cl}(X) = 1$.\end{proof}

\vspace{.1 in}
Since we have shown that $\pi_{cl}(X) = 1$ for every compact, one-dimensional geodesic space, the previous result implies the following.

\begin{corollary}
\label{1d covers}
Let $X$ be a compact, one-dimensional geodesic space.  Then $X$ has a universal cover if and only if it has a simply connected cover.
\end{corollary}

The following examples illustrate some local topological conditions that can prevent or allow the existence of a universal cover.  In particular, we point out that the isomorphism condition $\pi_1(X) \cong \bar{\pi}_1(X) \cong \Delta(X)$ that is necessary in Proposition \ref{charac simply conn covers} is not sufficient.  Example \ref{infinite torus} will show that these groups can be isomorphic even when the space has no simply connected or even universal cover.

\begin{example}
\label{HE example}
Let $H$ denote the geodesic Hawaiian earring.  This is an example of a space for which $\pi_1(H) \cong \bar{\pi}_1(H) \ncong \Delta(H)$, so that the necessary conditions at the end of Theorem \emph{\ref{main theorem 1}} are not satisfied.  Note that $H$ has no degenerate points, but the singular point where the circles are joined is sequentially singular.  Berestovskii-Plaut showed in \emph{\cite{BPUU}} that
\[\Delta(H) \cong \lim_{\leftarrow} F_n,\]
\noindent where $F_n$ is the rank $n$ free group and the groups form an inverse system via bonding homomorphisms $p_{nm}:F_m \rightarrow F_n$, $m > n$, that collapse the last $m-n$ generators to the identity and map the others to themselves.  Morgan and Morrison showed in \cite{MM} that $\pi_1(H)$ isomorphically embeds into this inverse limit of free groups, but its image is not the whole group \emph{(}de Smit, \cite{deS}\emph{)}.  In fact, an element of the inverse limit that is not be represented by a path loop in $H$ can be formed by a sequence of the commutators $[C_1,C_n]$, $n =1,2,\dots$  Thus, the UU-cover of the Hawaiian earring is not path connected and $\pi_1(H) \ncong \Delta(H)$.  On the other hand, since $H$ is one-dimensional, we have $\pi_{cl}(H) = 1$ and $\bar{\pi}_1(H) = \pi_1(H)$. $\blacksquare$
\end{example}

\begin{example}
\label{K example}
This space is a variant of one that is commonly used as a non-locally path connected example in more general topological settings (cf. \cite{FRVZ}).  The modification we use was introduced in \cite{SW1} by Sormani and Wei, and details may be found there. Define the following subsets of $\mathbb R^3$.
\[Z_1 = \bigl\{\bigl(x,y,\sin(1/y)\bigr) \ : \ 0 < y \leq 1/\pi, \ y \geq |x|\bigr\}\]  
\[Z_2 = \bigl\{(x,y,z) \ : \ y = |x|, \ -1/\pi \leq x \leq 1/\pi, \ -1 \leq z \leq 1\bigr\}\]
\noindent We then let $Z = Z_1 \cup Z_2$.  $Z$ is a topologist's sine curve ``wedge'' - or sine surface - with planar sides that come together at a vertical line segment at the apex of the wedge, namely $A \defeq \{(0,0,z) \ : \ -1 \leq z \leq 1\}$. With the subspace metric from $\mathbb R^3$, $Z$ is compact, connected, and locally path connected.  Obviously, any two points in $Z$ can be joined by a curve in $Z$ that is rectifiable with respect to the subspace metric, so we endow $Z$ with the corresponding induced geodesic metric.  With this geodesic metric, $Z$ is bi-Lipschitz equivalent to $Z$ with the subspace metric.  We call $A$ the \textit{apex of $Z$}.  Note that $Z \setminus A$ is semilocally simply connected, though not uniformly so.  We take as our base point $*$ the origin, which is the midpoint of $A$.

It is shown in \cite{SW1} that $Z$ is not homotopically Hausdorff - so $\pi_{cl}(Z)$ is not trivial -  and that $Z$ can be expressed as the Gromov-Hausdorff limit of a sequence of simply connected spaces, $Z_n$.  The latter property means that $Cr(Z) = \emptyset$, since a critical value of $Z$ would have to be a limit of critical values of the sequence spaces $Z_n$ (Theorem 8.4 in \cite{SW2}).  So, $Z$ is its own universal cover, and $\pi_1(Z) \ncong \bar{\pi}_1(Z) \cong \Delta(Z)$.  Note that $Z$ has no sequentially singular points, but it contains a continuum of degenerate points.

To obtain a similar example but with nontrivial covers, we can simply glue a circle of circumference $1$ - or any compact, semi-locally simply connected, non-simply connected space - to $Z$ at the base point $*$.  Let $Y=S^1 \vee Z$ be this wedge space.  Then the universal cover \emph{(}and the UU-cover\emph{)} of $Y$ is the real line with copies of $Z$ attached at the integers.  This space is path connected, so $\Lambda$ is surjective and $\bar{h}_{\Delta}:\bar{\pi}_1(Y) \rightarrow \Delta(Y)$ is an isomorphism.  The closed lifting group of $Y$ is still clearly nontrivial, but now it is not the whole group $\pi_1(Y)$.  Any nontrivial path loop at $*$ that stays in and traverses the circle in either direction will be unraveled at $\varepsilon = \frac{1}{3}$.  Thus, $\pi_1(Y) \ncong \bar{\pi}_1(Y)$.  Note that $Y$, like $Z$, has degenerate points but no sequentially singular points. $\blacksquare$\end{example}

The next example is of a space that has no simply connected or even universal cover but is such that all three groups $\pi_1(X)$, $\bar{\pi}_1(X)$, and $\Delta(X)$ are isomorphic and \textit{every} point is a sequentially singular point.  The space in this example is compact but infinite dimensional.

\begin{example}
\label{infinite torus}
Let $T^{\infty}$ be the infinite torus, metrized as a compact geodesic space.  Specifically, we let $T^{\infty}$ be the standard metric product \emph{(}i.e. $d = \sqrt{\sum_{n=1}^{\infty} d_n^2}$\emph{)} of the geodesic circles $S_n^1$, where $S_n^1$ has circumference $\frac{3}{2^n}$.  We fix base points $*_n \in S_n^1$ and choose $(*_1,*_2,\dots)$ as our base point in $T^{\infty}$.  Then $Cr(T^{\infty}) = \bigl\{\frac{1}{2^n}\bigr\}_{n=1}^{\infty}$, so it has no universal cover.  In fact, by the homogeneity of the space, every point in $T^{\infty}$ is a sequentially singular point.

It is straightforward to see that $\pi_{cl}(X) = 1$.  Indeed, every circle $S_n^1$ in the product eventually lifts open to $T_{\varepsilon}^{\infty}$ for small enough $\varepsilon$, and any nontrivial path loop in $T^{\infty}$ will have some nonzero power of such a circle as at least one of its factors.  Thus, $\Pi_{n=1}^{\infty} \mathbb Z \cong \pi_1(T^{\infty}) \cong \bar{\pi}_1(T^{\infty})$.  Note that $\Pi_{n=1}^{\infty} \mathbb Z$ is the direct product, not the direct sum.

It remains to compute $\Delta(T^{\infty})$, but this is straightforward.  Let $\varepsilon_n = \frac{1}{2^n}$, $n = 1, 2, \dots$  For any $\varepsilon_{n+1} < \varepsilon \leq \varepsilon_n$, the cover $T_{\varepsilon}^{\infty}$ simply unravels the first $n$ circles in the product, leaving the others unchanged.  Intuitively, the other circles are too small to be noticed by this $\varepsilon$-cover, and it views the space - in a coarse or large scale sense - as simply an $n$-dimensional torus.  Consequently, the $\varepsilon$-groups are easy to determine, also.  For $\varepsilon_{n+1} < \varepsilon \leq \varepsilon_n$, $\pi_{\varepsilon}(T^{\infty})$ is the rank $n$ free abelian group, $\mathbb Z^n$.  As $\varepsilon$ decreases and reaches a new critical value, the $\varepsilon$-groups simply pick up a new generator.  In other words, the bonding homomorphism $\Phi_{\varepsilon_n \varepsilon_{m}}:T_{\varepsilon_{m}}^{\infty} \rightarrow T_{\varepsilon_n}^{\infty}$, for $m > n$, are just the projection homomorphisms $\mathbb Z^m \rightarrow \mathbb Z^n$ that take the first $n$ generators to themselves and collapse the last $m-n$ generators to the identity.

Thus, $\Delta(T^{\infty})$ is isomorphic to the inverse limit of finite rank free abelian groups, $\lim_{\leftarrow} \mathbb Z^m$, with the bonding maps as just described.  This inverse limit, however, satisfies
\[\lim_{\leftarrow} \mathbb Z^m \cong \prod_{m=1}^{\infty} \mathbb Z.\]
\noindent So, we see that all three groups $\pi_1(T^{\infty})$, $\bar{\pi}_1(T^{\infty})$, and $\Delta(T^{\infty})$ are isomorphic. $\blacksquare$  
\end{example}

\begin{example}
\label{gasket example}
In this example, we show that $\pi_{cl}(X)$ need not equal $\pi_1^{sg}(X)$.  While the construction is technical, the basic idea is simple.  We take the standard Sierpinski gasket formed by an equilateral triangle, and we fill in the holes of the gasket with spaces similar to the space $Z$ from Example \ref{K example}.  Then the path loop that traverses the outer triangle of the gasket represents an element of $\pi_{cl}(X)$, but it cannot be expressed as a finite product of small loop lollipops.  We construct $X$ as a Hausdorff limit of compact subspaces of $\mathbb R^3$. The limit is a Peano continuum, yielding an equivalent geodesic metric by the Bing-Moise Theorem. Details on Gromov-Hausdorff limits may be found in \cite{BBI}.

To begin, we modify the space $Z$ from from Example \ref{K example} by reducing the diameter in the $y$-direction and expanding the angle between the rectangular sides from $90$ degrees to $120$.  Define $Z'$ to be the subset of $\mathbb R^3$ formed by the union of the following sets.
\[Z'_1 = \bigl\{\bigl(x,y,\sin(1/y)\bigr) : 0 < y \leq 1/3\pi, \ y \geq |x|/\sqrt{3}\bigr\}\]
\[Z'_2=\bigl\{(x,y,z) : y = |x|/\sqrt{3}, \ -1/\pi\sqrt{3} \leq x \leq 1/\pi\sqrt{3}, \ -1 \leq z \leq 1\bigr\}\]
\noindent Now, glue together three isometric copies of $Z'$ along their planar sides, so that the apexes of the copies are identified with each other.  This forms a subspace of $\mathbb R^3$ with an outer boundary curve given by an equilateral triangle of side length $\frac{2}{\pi\sqrt{3}} < \frac{1}{2}$ and with a vertical segment of height $2$ at the center that mimics the behavior of the apex in $Z$.

Next, attach a flat, equilateral triangular annulus to the outer boundary curve of this space, expanding the outer boundary curve to an equilateral triangle of side length $\frac{1}{2}$.  Call the resulting space $Z''$.  Let $S_2$ be the second stage graph approximation of the standard Sierpinski gasket formed from an equilateral triangle of side length $1$. That is, we let $T$ be the equilateral triangle of side length $1$, take a homothetic copy $T' = \frac{1}{2}T$, and then attach $T'$ to $T$ so that the vertices of $T'$ are at the midpoints of the sides of $T$.  Finally, we insert $Z''$ into the copy of $T'$ within $S_2$, identifying the outer boundary of $Z''$ with $T'$.  We will denote the resulting space by $X_1$, since it will be the first stage in a Gromov-Hausdorff sequence.

We endow $X_1$ with the subspace metric from $\mathbb R^3$.  Since the central vertical segment resulting from the sine surface portions has height $2$ - but height $1$ above and below the plane in which the outer boundary triangle lies - it is easy to verify that every point on the attached copy of $Z''$ is within a distance $\frac{\sqrt{13}}{2\sqrt{3}}$ of a point on the outermost triangular edge of $X_1$.  The rest of the construction now follows as in the standard gasket construction.  Let $p_1$, $p_2$, $p_3$ be the outer three vertices of $X_1$, starting with the lower left and moving  around counterclockwise.  Define contraction mappings $f_i(u) = \frac{1}{2}(u - p_i) + p_i$.  We set 
\[X_2  = X_1 \cup \Biggl(\bigcup_{i=1,2,3} f_i(X_1)\Biggr).\] 
\noindent That is, $X_2$ is constructed by taking three homothetic copies of $X_1$, rescaled by $\frac{1}{2}$, and attaching them to the holes of $X_1$, so that the boundaries of the contracted copies are identified with the boundaries of the holes.  This also rescales the height of the original, central vertical segment in $X_1$ by $\frac{1}{2}$, giving it a height of $1$.  Giving $X_2$ the Euclidean subspace metric, we see that $X_1$ isometrically embeds into $X_2$, since $X_1$ is not metrically altered by adding in the new pieces.  On the other hand, by simply rescaling by half the distance computations from $X_1$, we see that every point on $X_2$ is within a distance $\frac{\sqrt{13}}{2^2 \sqrt{3}}$ of a point of $X_1$.  Thus, $d_{GH}(X_1,X_2) < \frac{\sqrt{13}}{2^2 \sqrt{3}}$.

Continuing, we let $X_3 = X_2 \cup \bigl(\bigcup_{i,j=1,2,3} (f_i \circ f_j)(X_1)\bigr)$, which attaches homothetic copies of $X_1$, rescaled by $\frac{1}{4}$, to the holes in $X_2$.  The central vertical segments of the attached copies now have height $\frac{1}{2}$, so that every point of $X_3$ is within a distance $\frac{\sqrt{13}}{2^3 \sqrt{3}}$ of a point of $X_2$.  Since $X_2$ isometrically embeds into $X_3$, we have $d_{GH}(X_2,X_3) < \frac{\sqrt{13}}{2^3 \sqrt{3}}$.

Obviously, we can continue this process inductively, obtaining a sequence of compact metric spaces $\{X_n\}$ such that $d_{GH}(X_{n-1},X_n) < \frac{\sqrt{13}}{2^n \sqrt{3}}$ for $n \geq 2$.  This means that the resulting sequence is Cauchy, and since all of the spaces are metric subspaces of $\mathbb R^3$, $\{X_n\}$ is actually just a Cauchy sequence in the space of compact subsets of $\mathbb R^3$ with the Hausdorff metric.  That space is complete, which means $X_n$ converges to a compact metric space $X$.  It is evident that this limit space is simply the Sierpinski gasket with the holes filled in with appropriately rescaled copies of $Z''$.  Consequently, it admits a compatible geodesic metric by the Bing-Moise Theorem.  Moreover, like the original gasket, $X$ has a self-similarity property, namely that $X = \bigcup_{1\leq i \leq 3} f_i(X)$.  Note, also, that $Cr(X) = \emptyset$, since - by filling in the holes - the maximum critical value of $X_n$ converges to $0$ as $n \rightarrow \infty$.  Thus, $X$ is its own universal cover.

Choose the lower left vertex $p_1$ as the base point, and let $\gamma$ be the simple path loop that traverses the outer triangular boundary of $X$ one time in the counterclockwise direction.  Note that $\gamma$ is not a small path loop.  We claim that $[\gamma] \in \pi_{cl}(X)$ but is not in $\pi_1^{sg}(X)$.  This is intuitively evident, since $\gamma$ ``surrounds'' infinitely many separated degenerate points; the only formal obstacle is the bookkeeping.  Thus, for this example only, we will adopt the following simplified notation.  When a path loop $\sigma$ can be expressed as a product of lollipops,  $[\sigma] = [\alpha_1 \beta_1 \alpha_1^{-1} \cdots \alpha_k \beta_k \alpha_k^{-1}]$, we will ignore the adjoing paths, since they will not be essential in this particular example, and simply write $\sigma \sim \beta_1 \cdots \beta_k$.

We will denote a word of length $n$ in the set of symbols $\Sigma(3)\defeq \{1,2,3\}$ by $\omega = \{\omega_1,\dots,\omega_n\}$, where $\omega_i \in \Sigma(3)$ for each $i$.  For each such word $\omega = \{\omega_1,\dots,\omega_n\}$, define $\beta_{\omega}^{(n)} \defeq f_{\omega} \circ \gamma$, where $f_{\omega}$ is shorthand for $f_{\omega_1} \circ \cdots \circ f_{\omega_n}$.  For words of small length, we will often suppress the brace notation for $\omega=\{\omega_1,\dots,\omega_n\}$ when it is written as a subscript; thus, $f_{1,2}$ will denote $f_1 \circ f_2$, and so on.  Note that $\beta_{1}^{(1)} \defeq f_1 \circ \gamma$, $\beta_2^{(1)} \defeq f_2 \circ \gamma$, and $\beta_3^{(1)} \defeq f_3 \circ \gamma$ are just the three outer triangular path loops of $S_2$ considered as curves in $X$, starting with the lower left triangle and moving counterclockwise around $X$.  Now, let $\lambda^{(0)}$ denote the counterclockwise, triangular path loop determined by $T'$ in $S_2$, again considered as a path loop in $X$.  Then define $\lambda_{\omega}^{(n)} = f_{\omega} \circ \lambda^{(0)}$ for any $n \geq 1$ and word $\omega$ of length $n$.

It is known and easy to verify that, for the gasket, $\gamma$ is homotopic to a product of four lollipops, $\gamma \sim \beta_1^{(1)} \beta_2^{(1)} \lambda^{(0)} \beta_3^{(1)}$.  The same result holds for $X$, since the homotopy between $\gamma$ and this product of lollipops lies in $S_2$ and, thus, is not affected by inserting $Z''$ into the central hole of $S_2$.  By construction, the lollipop with head $\lambda^{(0)}$ is a small loop lollipop.  The other three are not, since $\beta_1^{(1)}$, $\beta_2^{(1)}$, and $\beta_3^{(1)}$ are images of $\gamma$ under $f_1$, $f_2$, and $f_3$ and, therefore, enclose homothetic copies of $X$.  For these other three, however, we can use the self-similarity and inductively apply the same homotopy decomposition we applied to $\gamma$, breaking up each $\beta_i^{(1)}$, $i=1,2,3$, into a product of four more smaller lollipops that mimic the original decomposition.  That is, one will be a small loop lollipop, and the other three that are not small can be further decomposed, yielding an inductive process.  

Consider $\beta_1^{(1)}$, for instance.  The homotopy used to decompose $\gamma$ can be transferred via $f_1$ to a homotopy in $f_1(X)$ that decomposes $\beta_1^{(1)}$ into a product of four lollipops
\[\beta_1^{(1)} \sim \beta_{1,1}^{(2)} \beta_{1,2}^{(2)} \lambda_1^{(1)} \beta_{1,3}^{(2)} =(f_1 \circ f_1 \circ \gamma) (f_1 \circ f_2 \circ \gamma) (f_1 \circ \lambda^{(0)}) (f_1 \circ f_3 \circ \gamma) .\]
\noindent Likewise, we have
\[\beta_2^{(1)} \sim \beta_{2,1}^{(2)} \beta_{2,2}^{(2)} \lambda_2^{(1)} \beta_{2,3}^{(2)} =(f_2 \circ f_1 \circ \gamma) (f_2 \circ f_2 \circ \gamma) (f_2 \circ \lambda^{(0)}) (f_2 \circ f_3 \circ \gamma)\]
\[\beta_3^{(1)} \sim \beta_{3,1}^{(2)} \beta_{3,2}^{(2)} \lambda_3^{(1)} \beta_{3,3}^{(2)} =(f_3 \circ f_1 \circ \gamma) (f_3 \circ f_2 \circ \gamma) (f_3 \circ \lambda^{(0)}) (f_3 \circ f_3 \circ \gamma) .\]
\noindent Note that the small loop lollipops, $\lambda_j^{(1)}$, are indexed differently, since they originate from $T'$ and not from $\gamma$.  We see, then, that
\[\gamma \sim \Bigl(\beta_{1,1}^{(2)} \beta_{1,2}^{(2)} \lambda_1^{(1)} \beta_{1,3}^{(2)}\Bigr)\Bigl(\beta_{2,1}^{(2)} \beta_{2,2}^{(2)} \lambda_2^{(1)} \beta_{2,3}^{(2)}\Bigr)\lambda^{(0)} \Bigl(\beta_{3,1}^{(2)} \beta_{3,2}^{(2)} \lambda_3^{(1)} \beta_{3,3}^{(2)}\Bigr).\]
\noindent Then, we continue the decomposition process inductively at the next level.  We leave the small loop lollipops as they are.  Each $\beta_{i,j}^{(2)}$, however, is decomposed into a product $\beta_{i,j,1}^{(3)} \beta_{i,j,2}^{(3)} \lambda_{i,j}^{(2)} \beta_{i,j,3}^{(3)}$, where $\lambda_{i,j}^{(2)}$ is small and the other three path loops can be further decomposed.

Since the triangles making up the heads of the lollipops are either small path loops or become arbitrarily small in diameter as $n \rightarrow \infty$, the end result of the above decomposition process is that $\gamma$ can be expressed as a product of $\varepsilon$-lollipops for any $\varepsilon > 0$.  Thus, $[\gamma] \in \pi_{cl}(X)$.  However, $\gamma$ cannot be expressed as a finite product of small loop lollipops.  In fact, as the decomposition process above shows, any finite product of lollipops representing $[\gamma]$ will contain at least one that is not small but can be further decomposed using the self-similarity of $X$.  Hence, $[\gamma]$ is not in $\pi_1^{sg}(X)$.$\blacksquare$
\end{example}

\vspace{.1 in}
\section{Universal Covers and Group Topology}

If $I=[0,1] \subset \mathbb R$, the pointed, continuous mapping space 
\[Hom\bigl((I,\{0,1\}),(X,*)\bigr)=\{f:I \rightarrow X \ | \ f \ cts., \ f(0)=f(1)=*\}\]
\noindent is given the uniform topology, and the \textit{topological fundamental group} is $\pi_1(X,*)$ topologized as the quotient of that space under the homotopy equivalence relation.  It is denoted by $\pi_1^{top}(X,*)$, and we refer to its topology as the \textit{uniform quotient topology}.  We will let $\bar{\pi}_1^{top}(X,*)$ denote the revised fundamental group with the quotient topology inherited from $\pi_1^{top}(X,*)$, and we call this the \textit{topological revised fundamental group}.  As usual, we often suppress the base point $*$.

We first show that path loops that are sufficiently uniformly close, in terms of a given $\varepsilon > 0$, represent the same element of $\pi_1(X)/K_{\varepsilon}$.

\begin{lemma}
\label{close path lifts}
Let $X$ be a geodesic space with base point $*$, and suppose $\gamma$ and $\lambda$ are path loops at $*$.  If $\sup_{0\leq t \leq 1} d(\gamma(t),\lambda(t)) < \frac{\varepsilon}{3}$, then the lifts of $\gamma$ and $\lambda$ to $X_{\varepsilon}$ are either both closed or both open, and $[\gamma]K_{\varepsilon} = [\lambda]K_{\varepsilon}$.
\end{lemma}

\begin{proof}
Choose a strong $\frac{\varepsilon}{3}$-loop along $\gamma$, say $\alpha = \{*=\gamma(0),\gamma(t_1),\dots,\gamma(t_n)=*\}$.  Define a chain loop $\beta =\{*=\lambda(0),\lambda(t_1),\dots,\lambda(t_n)=*\}$ along $\lambda$.  Then $\mathcal D(\alpha,\beta) < \frac{\varepsilon}{3}$, and $E(\alpha) > \frac{2\varepsilon}{3}$.  It follows that $\mathcal D(\alpha,\beta) < \frac{1}{2}E(\alpha)$, and, by Lemma \ref{close chain lemma}, $\beta$ is an $\varepsilon$-loop that is $\varepsilon$-homotopic to $\alpha$.  Furthermore, $\beta$ is a strong $\varepsilon$-loop along $\lambda$.
In fact, if $i\in \{1,\dots,n\}$ and $t_{i-1} \leq t \leq t_i$, then
\[d(\lambda(t_{i-1}),\lambda(t))\leq d(\lambda(t),\gamma(t)) + d(\gamma(t),\gamma(t_{i-1})) + d(\gamma(t_{i-1}),\lambda(t_{i-1})) < \varepsilon\]
\[d(\lambda(t_{i}),\lambda(t))\leq d(\lambda(t),\gamma(t)) + d(\gamma(t),\gamma(t_{i})) + d(\gamma(t_{i}),\lambda(t_{i})) < \varepsilon.\]
\noindent So, if $\gamma$ lifts closed to $X_{\varepsilon}$, then $\alpha$ is $\varepsilon$-null.  This implies that $\beta$ is also $\varepsilon$-null, and $\lambda$ lifts closed to $X_{\varepsilon}$.  If $\gamma$ lifts open to $X_{\varepsilon}$, then $\alpha$ is $\varepsilon$-nontrivial.  Hence, $\beta$ is also $\varepsilon$-nontrivial, and $\lambda$ lifts open. For the last part, consider the strong $\varepsilon$-loop $\alpha \beta^{-1}$ along $\gamma \lambda^{-1}$.  Since $\alpha \beta^{-1}$ is $\varepsilon$-null, $[\gamma] [\lambda]^{-1}=[\gamma \lambda^{-1}] \in K_{\varepsilon}$.\end{proof}

\begin{corollary}
\label{kernels clopen}
If $X$ is a geodesic space, then $K_{\varepsilon} \subset \pi_1(X)$ is both open and closed in $\pi_1^{top}(X)$.
\end{corollary}

\begin{proof}
Let $p:Hom\bigl((I,\partial I),(X,*)\bigr) \rightarrow \pi_1^{top}(X)$ be the quotient map.  To show that $K_{\varepsilon}$ is open, we need to show that $p^{-1}(K_{\varepsilon})$ is open in $Hom\bigl((I,\partial I),(X,*)\bigr)$.  But an element of $p^{-1}(K_{\varepsilon})$ is a path loop at $*$ that lifts closed to $X_{\varepsilon}$, and - by Lemma \ref{close path lifts} - any path loop at $*$ that is uniformly within $\frac{\varepsilon}{3}$ of such a loop will also lift closed to $X_{\varepsilon}$.  In other words, if $\gamma \in p^{-1}(K_{\varepsilon})$, then the open uniform ball $B_{\varepsilon/3}(\gamma)$ is also in $p^{-1}(K_{\varepsilon})$, showing that $p^{-1}(K_{\varepsilon})$ is open.  Finally, the complement of $K_{\varepsilon}$ is $\bigcup_{g \notin K_{\varepsilon}} gK_{\varepsilon}$.  The cosets $gK_{\varepsilon}$ are open, because left translations are homeomorphisms in the quasitopological group $\pi_1^{top}(X)$.\end{proof}

\vspace{.1 in}
\noindent It follows, then, that $\pi_{cl}(X) = \bigcap_{\varepsilon > 0} K_{\varepsilon}$ is closed in $\pi_1^{top}(X)$.

\begin{lemma}
\label{ss point not discrete}
If a geodesic space $X$ has a sequentially singular point, then $\bar{\pi}_1^{top}(X)$ is not discrete.
\end{lemma}

\begin{proof}
We may assume without loss of generality that the base point $*$ is a sequentially singular point, and it suffices to show that $\pi_{cl}(X,*) = \bar{h}^{-1}(1_{\bar{\pi}})$ is not open, where $1_{\bar{\pi}}=\pi_{cl}(X)$ is the identity in $\bar{\pi}_1^{top}(X)$.  Let $r_n \searrow 0$ and $\gamma_n$ be sequences of positive real numbers and path loops at $*$ as in the definition of sequentially singular, and let $p:Hom((I,\partial I),(X,*)) \rightarrow \pi_1^{top}(X)$ be the quotient map.  If $V$ is any open set in $\pi_1^{top}(X)$ containing the identity $1_{\pi} \in \pi_1^{top}(X)$, then $p^{-1}(V)$ is an open set containing the constant path $*$, which means that it contains a uniform open ball of some radius $\tau$ centered at the constant path $*$.  We can choose $N \in \mathbb N$ so that $n \geq N \Rightarrow r_n < \tau$.  For such $n$, $\gamma_n$ lies in $B_{\tau}(*) \subset p^{-1}(V)$, and, thus,  $[\gamma_n]=p(\gamma_n) \in V$.  But each $\gamma_n$, by definition, eventually lifts open to $X_{\varepsilon}$ for sufficiently small $\varepsilon$, showing that every open set containing $1_{\pi}$ will contain elements not in $\pi_{cl}(X)$.  Therefore, $\pi_{cl}(X)$ is not open.\end{proof}

\begin{proof}[\textbf{Proof of Theorem \emph{\ref{uc cover eq discrete group}}}]
We have already noted that $\pi_{cl}(X)$ is closed, and the quotient of a quasitopological group by a closed, normal subgroup is a $T_1$ quasitopological group (cf. Theorems 1.5.1 and 1.5.3, \cite{AATM}).  If $\bar{\pi}_1^{top}(X)$ is discrete, then $X$ cannot have any sequentially singular points by Lemma \ref{ss point not discrete}.  Thus, it has a universal cover by Theorem \ref{main theorem 1}.  If $X$ has a universal cover, then there is some $\varepsilon > 0$ such that $\pi_{cl}(X) = K_{\varepsilon}$.  This means that for any element $g\pi_{cl}(X) \in \bar{\pi}_1^{top}(X)$, $\bar{h}^{-1}(g\pi_{cl}(X))$ equals the open set $gK_{\varepsilon}$.  Hence, each point in $\bar{\pi}_1^{top}(X)$ is open, and the revised fundamental group is discrete.\end{proof}

\vspace{.1 in}
Next, we introduce another topology on the fundamental group and examine its relationship to the existence of universal and simply connected covers. The \textit{subgroup topology} on a group $G$ determined by a collection, $\Sigma$, of subgroups of $G$ is defined as follows (see Section 2.5 of \cite{BS}).  The collection $\Sigma$ is a \textit{neighborhood family} if for any $H_1$, $H_2 \in \Sigma$, there is a subgroup $H_3 \in \Sigma$ such that $H_3 \subset H_1 \cap H_2$.  As a neighborhood family, the collection of all left cosets of subgroups in $\Sigma$ forms a basis for a topology on $G$, called the \textit{subgroup topology determined by $\Sigma$}.  This follows because the intersection of left cosets of subgroups $H_1$ and $H_2$ is a left coset of the intersection $H_1 \cap H_2$.  Thus, if $g \in g_1H_1 \cap g_2 H_2$, then there is some $g'$ such that $g \in g_1H_1 \cap g_2H_2 = g'(H_1 \cap H_2)$, which means that we can take $g' = g$.  Taking $H_3 \subset H_1 \cap H_2$, we then see that $g \in gH_3 \subset g(H_1 \cap H_2) = g_1H_1 \cap g_2H_2$.  

It is pointed out in \cite{BS} that $G$ with this topology need not be a topological group, in general.  We will show, however, that it is a topological group when the subgroups in $\Sigma$ are normal.

\begin{lemma}
\label{sg top group}
Let $G$ be a group endowed with the subgroup topology $\mathcal T_{\Sigma}$ determined by a neighborhood family of subgroups $\Sigma$.  If the subgroups in $\Sigma$ are normal, then $(G,\mathcal T_{\Sigma})$ is a topological group.
\end{lemma}

\begin{proof}
We first show that multiplication is continuous.  Let $g_1$, $g_2 \in G$ be given.  If $U$ is open and $g_1g_2 \in U$, then there is a basis element, $gH$, $H \in \Sigma$, such that $g_1g_2 \in gH \subset U$.  This means we can take $g = g_1 g_2$.  Suppose $(h_1,h_2) \in g_1H \times g_2H$.  By normality, we have $g_1H = Hg_1$. Thus, $(h_1,h_2) \in Hg_1 \times g_2H$, and we can find $k_1$, $k_2 \in H$ such that $h_1 = k_1g_1$ and $h_2 = g_2 k_2$.  It follows that $h_1h_2 = k_1g_1g_2k_2$.  But, again using normality, $k_1g_1g_2k_2 \in Hg_1g_2k_2 = g_1g_2k_2H = g_1g_2H$, so $h_1 h_2 \in g_1g_2H \subset U$.

Next, let $g \in G$ be given.  If $U$ is open and contains $g^{-1}$, then there is some $H\in \Sigma$ such that $g^{-1} \in g^{-1}H \subset U$.  So, if $h \in gH=Hg$, there is $k \in H$ so that $h=kg$.  But then $h^{-1} = g^{-1}k^{-1} \in g^{-1}H$.\end{proof}

\vspace{.1 in}
\noindent We will also need the following result, which is part of a general subgroup topology theorem proved by Bogley and Sierdaski in \cite{BS}.  For reference purposes, we state the needed portion here as a lemma.

\begin{lemma}[\textbf{Theorem 2.9, \cite{BS}}]
\label{BS lemma}
Let $G$ be a group with the subgroup topology determined by a neighborhood family of subgroups, $\Sigma$, and let $H_{\Sigma} = \bigcap_{H \in \Sigma} H$.  Then the connected component containing $g \in G$ is $gH_{\Sigma}$, and $G$ is discrete if and only if $\Sigma$ contains the trivial subgroup.  Moreover, the following are equivalent.  \emph{1)} $G$ is $T_0$; \emph{2)} $G$ is Hausdorff; \emph{3)} $H_{\Sigma}$ is trivial; \emph{4)} $G$ is totally disconnected.
\end{lemma}

Now, we assume that $X$ is a compact geodesic space, and we consider the collection $\Sigma_{\kappa} = \{K_{\varepsilon}\}_{\varepsilon > 0}$.  This collection is a particularly simple neighborhood family, since it is actually a countable, directed set: $\delta < \varepsilon$ implies that $K_{\delta} \subset K_{\varepsilon}$, and the only changes occur at a discrete, strictly decreasing set of values, namely the critical values.  In fact, given $K_{\varepsilon_1}$, $K_{\varepsilon_2} \in \Sigma_{\kappa}$, we may assume without loss of generality that $\varepsilon_1 \leq \varepsilon_2$, in which case $K_{\varepsilon_1} \cap K_{\varepsilon_2} = K_{\varepsilon_1}$.  Note that $H_{\Sigma_{\kappa}} = \bigcap_{\varepsilon > 0}K_{\varepsilon} = \pi_{cl}(X)$.

\begin{definition}
\label{ctfg}
Let $\pi_1^{\mathcal C}(X)$ denote the fundamental group of $X$ with the subgroup topology determined by $\Sigma_{\kappa}$.  We call this topology the covering topology on $\pi_1(X)$, and we call $\pi_1^{\mathcal C}(X)$ the $\mathcal C$-fundamental group.  The quotient group $\bar{\pi}_1^{\mathcal C}(X) \defeq \pi_1^{\mathcal C}(X)/\pi_{cl}(X)$ with the quotient topology from $\pi_1^{\mathcal C}(X)$ is the revised $\mathcal C$-fundamental group.
\end{definition}

By Lemma \ref{sg top group}, $\pi_1^{\mathcal C}(X)$ is a topological group, and, by Lemma \ref{BS lemma}, it is Hausdorff if and only if $\pi_{cl}(X) = 1$.  As basis elements, the subgroups $K_{\varepsilon}$ are obviously open and, thus, also closed in $\pi_1^{\mathcal C}(X)$.  Hence, $\pi_{cl}(X)$ is closed in $\pi_1^{\mathcal C}(X)$.  Moreover, we see from Lemma \ref{BS lemma} that the connected component in $\pi_1^{\mathcal C}(X)$ containing $g$ is the coset $g\pi_{cl}(X)$.  These statements, along with the topological group structure of $\pi_1^{\mathcal C}(X)$, immediately imply the following.

\begin{lemma}
\label{sg quotient tg}
If $X$ is a compact geodesic space, then $\bar{\pi}_1^{\mathcal C}(X)$ is a Hausdorff topological group, and its elements are in bijective correspondence with the connected components of $\pi_1^{\mathcal C}(X)$.
\end{lemma}

Note, also, that the topological group structure of $\pi_1^{\mathcal C}(X)$ shows that the covering and uniform quotient topologies on $\pi_1(X)$ need not coincide in general.  As we mentioned in the introduction, there are examples of compact geodesic spaces for which $\pi_1^{top}(X)$ is not a topological group, while $\pi_1^{\mathcal C}(X)$ always is.  However, since every $gK_{\varepsilon} \subset \pi_1(X)$ is a basis element in the covering topology and open in the uniform quotient topology, every open set in the former topology is also open in the latter.  Thus, the uniform quotient topology is always at least finer than the covering topology on $\pi_1(X)$.  Furthermore, we have the following.

\begin{lemma}
\label{topologies equal}
If $X$ is a compact geodesic space with a universal cover, and if $\pi_{cl}(X) = \pi_1^{sg}(X)$, then the covering and uniform quotient topologies on $\pi_1(X)$ are equal.
\end{lemma}

\begin{proof}
We need to show that every open set in $\pi_1^{top}(X)$ is open in $\pi_1^{\mathcal C}(X)$.  Let $U \subset \pi_1(X)$ be open in the uniform quotient topology, and let $g=[\gamma] \in U$ be given.  We need to find $\varepsilon > 0$ such that $g \in gK_{\varepsilon} \subset U$.

Since $X$ has a universal cover, $Cr(X)$ is finite, and we choose $\varepsilon < \min Cr(X)$.  Then $K_{\varepsilon}=\pi_{cl}(X)=\pi_1^{sg}(X)$.  Let $[\lambda]$ be any element in $gK_{\varepsilon}$. Then there is a finite product of small path loops, say $\alpha_1 \beta_1 \alpha_1^{-1} \cdots \alpha_n \beta_n \alpha_n^{-1}$, such that $[\lambda] = [\gamma \alpha_1 \beta_1 \alpha_1^{-1} \cdots \alpha_n \beta_n \alpha_n^{-1}]$.  We assume that $\gamma \alpha_1 \beta_1 \alpha_1^{-1} \cdots \alpha_n \beta_n \alpha_n^{-1}$ is parameterized on $[0,1]$ so that $\gamma$ is parameterized on $[0,1/2]$, $\alpha_i$ is parameterized on $\bigl[\frac{1}{2}+\frac{3(i-1)}{6n},\frac{1}{2}+\frac{3i-2}{6n}\bigr]$, $\beta_i$ is parameterized on $\bigl[\frac{1}{2} + \frac{3i-2}{6n},\frac{1}{2}+\frac{3i-1}{6n}\bigr]$, and $\alpha_i^{-1}$ is parameterized on $\bigl[ \frac{1}{2} + \frac{3i-1}{6n},\frac{1}{2} + \frac{3i}{6n}\bigr]$, for $i=1,\dots,n$.

Now, let $\gamma'$ denote the path loop $\gamma \alpha_1 \sigma_1 \alpha_1^{-1} \cdots \alpha_n \sigma_n \alpha_n^{-1}$, where $\sigma_i$ is the constant path at the endpoint of $\alpha_i$.  We asssume that $\gamma'$ is parameterized in the same way as $\gamma \alpha_1 \beta_1 \alpha_1^{-1} \cdots \alpha_n \beta_n \alpha_n^{-1}$, just with $\sigma_i$ replacing $\beta_i$. Clearly $\gamma'$ is homotopic to $\gamma$, so that $\gamma' \in p^{-1}(U) \subset Hom((I,\partial I),(X,*))$.  Thus, there is some $\delta > 0$ such that the open uniform ball $B_{\delta}(\gamma')$ is contained in $p^{-1}(U)$.  Since each $\beta_i$ is a small path loop, we can contract each one so that it lies in the open ball of radius $\delta$ centered at the endpoint of $\alpha_i$.  Doing so while leaving $\gamma$ and each $\alpha_i$ fixed, and maintaining the parameterization of described above, we obtain a representative of $[\lambda]= [\gamma \alpha_1 \beta_1 \alpha_1^{-1} \cdots \alpha_n \beta_n \alpha_n^{-1}]$ that is uniformly within $\delta$ of $\gamma'$.  Thus, there is a path loop $\lambda' \in [\lambda]$ that lies in $p^{-1}(U)$, which means that $[\lambda] \in U$.  This shows that $gK_{\varepsilon} \subset U$, and $U$ is open in the covering topology.\end{proof}

\begin{example}
For the spaces in Example \ref{K example}, the critical and uniform quotient topologies agree.  By Proposition \ref{charac simply conn covers}, the topologies also agree when $X$ is compact and semilocally simply connected.
\end{example}

We can now prove Theorem \ref{sgtop thm1}.

\begin{proof}[\textbf{Proof of Theorem \ref{sgtop thm1}}]
By Lemma \ref{BS lemma}, $\pi_1^{\mathcal C}(X)$ is discrete if and only if $\Sigma_{\kappa}$ contains the trivial subgroup.  If $\pi_1^{\mathcal C}(X)$ is discrete, this means that $K_{\varepsilon} = 1$ for some $\varepsilon$, which further means that $K_{\tau} = 1$ for all $0 < \tau \leq \varepsilon$.  Hence, $Cr(X)$ is finite and $\pi_{cl}(X) = 1$, and the result follows from Proposition \ref{charac simply conn covers}.  On the other hand, if $X$ has a simply connected cover, then $K_{\varepsilon}$ is eventually trivial for small enough $\varepsilon$, also by Proposition \ref{charac simply conn covers}.  But this implies that $\Sigma_{\kappa}$ contains the trivial group, and $\pi_1^{\mathcal C}(X)$ is discrete.  This proves the simply connected covering portion of the theorem.

Suppose $X$ has a universal cover. Then $\pi_{cl}(X) = K_{\varepsilon}$ for small enough $\varepsilon$.  But this means that $\bar{h}^{-1}(1_{\bar{\pi}}) = \pi_{cl}(X)$ is open in $\pi_1^{\mathcal C}(X)$, and $\bar{\pi}_1^{\mathcal C}(X)$ is discrete.

Lastly, suppose $X$ does not have a universal cover.  Then it has a sequentially singular point, which we can assume is the base point $*$.  Let $U$ be any neighborhood of the identity $1_{\pi} \in \pi_1^{\mathcal C}(X)$ with $1_{\pi} \in K_{\varepsilon} \subset U$.  If $\gamma_n$ and $r_n \searrow 0$ are as in the definition of sequentially singular, then, once $r_n < \varepsilon$, each $\gamma_n$ lifts closed to $X_{\varepsilon}$.  Thus, $[\gamma_n] \in K_{\varepsilon} \subset U$ for all but finitely many $n$, showing that any neighborhood of the identity $1_{\pi} \in \pi_1^{\mathcal C}(X)$ will contain elements not in $\pi_{cl}(X)$. But, similarly to the proof of Lemma \ref{ss point not discrete}, this means that $\pi_{cl}(X)=\bar{h}^{-1}(1_{\bar{\pi}})$ is not open.  Hence, $\bar{\pi}_1^{\mathcal C}(X)$ is not discrete.\end{proof}

\vspace{.1 in}
It is useful to further compare and contrast the covering and uniform quotient topologies on $\pi_1(X)$.  We have already noted that the covering topology has the advantage of always making $\pi_1(X)$ a topological group.  Comparing Theorem \ref{sgtop thm1} to Theorem \ref{uc cover eq discrete group} and the aforementioned result of Fabel that $X$ has a simply connected cover if and only if $\pi_1^{top}(X)$ is discrete, we see that the topologies are equally effective in determining when $X$ has a universal or simply connected cover.  Consequently, there is also no difference between the topologies in detecting the presence of singular points.  For example, if $X$ has a sequentially singular point, then $X$ has no universal cover, and all four groups $\pi_1^{top}(X)$, $\bar{\pi}_1^{top}(X)$, $\pi_1^{\mathcal C}(X)$, and $\bar{\pi}_1^{\mathcal C}(X)$ are non-discrete.  Conversely, if either $\bar{\pi}_1^{top}(X)$ or $\bar{\pi}_1^{\mathcal C}(X)$ is non-discrete, then $X$ has a sequentially singular point.  If either $\pi_1^{top}(X)$ or $\pi_1^{\mathcal C}(X)$ is non-discrete, then $X$ has no simply connected cover, which - by Proposition \ref{charac simply conn covers} - implies one of two mutually exclusive possibilities:  either $Cr(X)$ is infinite and $X$ has a sequentially singular point, or $X$ has a non-simply connected universal cover and, thus, a degenerate point but no sequentially singular points. 

There is, however, one important advantage offered by the covering topology, in addition to the topological group structure.  Namely, this topology provides a useful geometric picture of $\pi_1^{\mathcal C}(X)$ and its relationship to $\bar{\pi}_1^{\mathcal C}(X)$, which further leads to conditions under which $\pi_{cl}(X)$ is trivial.  This brings us to Theorem \ref{sgtop thm2}.  The first statement of Theorem \ref{sgtop thm2} and the equivalence of $1$, $2$, and $3$ follow directly from Lemma \ref{BS lemma}.  The last statement of Theorem \ref{sgtop thm2} follows from part 1 and our result that $\bar{\pi}_1(X)$ injects into the first shape group.  All that remains is to prove the equivalence of $4$ and $5$ to the other conditions.  Before proceeding, some further remarks on geometric groups are in order.

As we mentioned in the introduction, the existence of a geometry on a topological group $G$ dictates certain topological conditions that must hold for $G$, and the particular semigroup $S$ indexing the geometry determines what kind of metrics induce the given topology.  For example, assume $G$ is a first countable, Hausdorff topological group.  If $G$ is also complete, then it admits a geometry $\{U_s\}$ over $S=\mathbb R_+$ - the positive reals with usual addition as the semigroup operation - if and only if it is path connected and locally path connected and admits a length metric that is compatible with the original topology (Proposition 3.1, remarks following Example 1.8 in \cite{BPGG}).  On the other hand, $G$ admits a geometry over $\mathbb R^{max}$ if and only if it admits a left-invariant \textit{ultrametric} compatible with the original topology, and this holds if and only if $G$ is totally disconnected (Example 1.6 and Section 3 in \cite{BPGG}).  Recall that an ultrametric satisfies the stronger triangle inequality $d(x,z) \leq \max\{d(x,y),d(y,z)\}$ for all $x,y,z$, and ultrametric spaces are necessarily totally disconnected.

Now, consider $\pi_1^{\mathcal C}(X)$ and the collection of open subgroups $\Sigma_{\kappa} = \{K_{\varepsilon}\}_{\varepsilon > 0}$.  We take $\Sigma_{\kappa}$ indexed over $\mathbb R^{max}$.  Clearly this collection is a local basis at the identity, simply by definition.  We already know that $\delta < \varepsilon$ if and only if $K_{\delta} \subseteq K_{\varepsilon}$, since the kernels form a nested, decreasing collection.  Likewise, we have $K_{\delta}K_{\varepsilon} = K_{\max\{\delta,\varepsilon\}}$, since $\delta \leq \varepsilon$ would imply that $K_{\delta} \subseteq K_{\varepsilon} \Rightarrow K_{\delta}K_{\varepsilon} = K_{\varepsilon}$, and similarly if $\varepsilon < \delta$.  When $X$ is compact, $K_{\varepsilon} = \pi_1(X)$ for any $\varepsilon > diam(X)$, since $X_{\varepsilon}$ is the trivial cover in that case.  Thus, $\bigcup_{\varepsilon > 0} K_{\varepsilon} = \pi_1(X)$.  Finally, each $K_{\varepsilon}$ is symmetric since it is a subgroup.

Thus, we see that $\Sigma_{\kappa}$ satisfies all of the conditions for a geometry over $\mathbb R^{max}$, except, perhaps, the condition $\pi_{cl}(X) = \bigcap_{\varepsilon > 0} K_{\varepsilon} = 1$.

\begin{proof}[\textbf{Proof of Theorem \ref{sgtop thm2}}]
We first note that $\pi_1^{\mathcal C}(X)$ is first countable.  In fact, $\Sigma_{\kappa}$ is a countable collection, for if $Cr(X) = \{\varepsilon_n\}$, with $\{\varepsilon_n\}$ either finite or strictly decreasing to $0$, then the only distinct elements of $\Sigma_{\kappa}$ are $K_{\varepsilon_n}$.  So, if $g$ is any element in $\pi_1^{\mathcal C}(X)$ and $U$ is any open set containing $g$, then there is $\varepsilon > 0$ such that $g \in gK_{\varepsilon} \subset U$.  But there is some $n$ such that $K_{\varepsilon} = K_{\varepsilon_n}$.  Hence, the collection $\{gK_{\varepsilon_n}\}$ forms a countable local basis at $g$.

We only need to prove that $1 \Leftrightarrow 4 \Rightarrow 5 \Rightarrow 2$.  That $1$ and $4$ are equivalent follows from the previous discussion and the definition of a geometry. If $4$ holds, then $\pi_1^{\mathcal C}(X)$ is a Hausdorff, first countable topological group admitting a geometry over $\mathbb R^{max}$, which means that $5$ holds by the results of Berestovskii-Plaut-Stallman cited above.  Finally, if $5$ holds, then $\pi_1^{\mathcal C}(X)$ is totally disconnected.\end{proof}

\begin{example}
\label{1d example}
Any one-dimensional, compact geodesic space, $X$, satisfies $\pi_{cl}(X) = 1$.  Thus, for such a space, $\pi_1^{\mathcal C}(X)$ is a totally disconnected, ultrametrizable, geometric group.  In particular, this holds for the Hawaiian earring and fractals like the Sierpinski carpet and gasket. $\blacksquare$
\end{example}

It is natural to wonder whether or not the covering topology may be generalized by considering other subgroup topologies on $\pi_1(X)$ determined by normal subgroups.  For example, one possibility is to consider the general \textit{Spanier groups} of $X$.  The Spanier subgroup determined by an open covering $\mathcal U$ of $X$ is the normal subgroup $\pi_1^{\mathcal U}(X) \subset \pi_1(X)$ generated by lollipops $\alpha \beta \alpha^{-1}$ where $\beta$ is a path loop lying in an element of $\mathcal U$.  The subgroups $K_{\varepsilon}$, for example, are Spanier subgroups.  In fact, we pointed out in Remark \ref{SW connection} that $K_{\varepsilon}$ equals the Spanier subgroup determined by the open covering of $\frac{3\varepsilon}{2}$-balls.  

In the more general topological setting, the Spanier subgroups may provide a means of determining topologies on $\pi_1(X)$ alternative to the compact-open quotient topology.  In the case of the covering topology on a compact geodesic case, however - where every open cover is refined by a uniform open cover of metric balls - this particular level of generalization turns out to be redundant, as the following shows.  From another point of view, though, this result also shows that the covering topology is determined by a deeper structure than just a geodesic metric and the critical spectrum; it is a more general topology that derives from a very natural construction.

\begin{lemma}
Let $X$ be compact geodesic, and let $\Omega(X)$ be the collection of open coverings of $X$.  Let $\Sigma_{\Omega}$ be the collection of Spanier subgroups of $X$ determined by the open coverings in $\Omega(X)$.  Then $\Sigma_{\Omega}$ is a neighborhood family, and the topology on $\pi_1(X)$ determined by $\Sigma_{\Omega}$ is equivalent to the covering topology.
\end{lemma}

\begin{proof}
We first note that if $\mathcal U$, $\mathcal V \in \Omega(X)$ and $\mathcal V$ is a refinement of $\mathcal U$, then $\pi_1^{\mathcal V}(X) \subseteq \pi_1^{\mathcal U}(X)$, for if $\alpha \beta \alpha^{-1}$ is a lollipop with $\beta$ lying in an element of $V \in \mathcal V$, then there is some $U \in \mathcal U$ such that $V \subset U$. But this means that $\alpha \beta \alpha^{-1}$ is a $U$-lollipop, also.  So, suppose $\pi_1^{\mathcal U}(X)$, $\pi_1^{\mathcal V}(X) \in \Sigma_{\Omega}$, and let $\mathcal W = \mathcal V \cap \mathcal U = \{W\subset X : W = V \cap U, \ V \in \mathcal V, \ U \in \mathcal U\}$.  Then $\mathcal W$ is an open covering of $X$, and it is a refinement of both $\mathcal U$ and $\mathcal V$.  It follows that $\pi_1^{\mathcal W}(X) \subset \pi_1^{\mathcal V}(X)$ and $\pi_1^{\mathcal W}(X) \subset \pi_1^{\mathcal U}(X)$, which further implies that $\pi_1^{\mathcal W}(X) \subset \pi_1^{\mathcal U}(X) \cap \pi_1^{\mathcal V}(X)$.  This shows that $\Sigma_{\Omega}$ is a neighborhood family.

We let $\pi_1^{\mathcal S}(X)$ denote $\pi_1(X)$ with the subgroup topology determined by $\Sigma_{\Omega}$, and we call this topology the Spanier topology.  Note that since the Spanier groups are normal, $\pi_1^{\mathcal S}(X)$ is a topological group.

If $A \subset \pi_1(X)$ is open in the covering topology and $g \in A$, then there is some $K_{\varepsilon}$ such that $g \in gK_{\varepsilon} \subset A$.  But $K_{\varepsilon}$ is in $\Sigma_{\Omega}$, so $gK_{\varepsilon}$ is a basis element in the Spanier topology.  Thus, $A$ is open in $\pi_1^{\mathcal S}(X)$.  Conversely, suppose $A \subset \pi_1(X)$ is open in the Spanier topology, and let $g \in A$ be given.  There is some $\mathcal U \in \Omega(X)$ such that $g \in g\pi_1^{\mathcal U}(X) \subset A$.  Now, since $X$ is compact, there is some $\varepsilon > 0$ such that every open ball of radius $\varepsilon$ is contained in some element of $\mathcal U$.  Let $\mathcal V$ be the open covering of $X$ consisting of all open balls of radius $\varepsilon$, so that $\mathcal V$ is a refinement of $\mathcal U$.  Then $\pi_1^{\mathcal V}(X) \subset \pi_1^{\mathcal U}(X)$, and $\pi_1^{\mathcal V}(X)$ is just $K_{2\varepsilon/3}$.  It follows that $g \in gK_{2\varepsilon/3} = g\pi_1^{\mathcal V}(X) \subset g\pi_1^{\mathcal U}(X) \subset A$, and $A$ is open in the covering topology.\end{proof}

\addtocontents{toc}{\protect\setcounter{tocdepth}{0}}
\subsection*{Acknowledgements}
\addtocontents{toc}{\protect\setcounter{tocdepth}{1}}
The author would like to thank Conrad Plaut and Christina Sormani for their careful reading and helpful suggestions, Ziga Virk for sharing information on small loops, and the mathematics department at The University of Connecticut for their support.  Some of this work was completed while the author was a doctoral candidate at The University of Tennessee.

\end{document}